\documentclass[12pt]{article}
\usepackage{mathtools}
\usepackage{appendix}
\usepackage{bm}

\usepackage{amsmath,latexsym,amsfonts,enumerate}
\usepackage{amssymb,amsthm}
\usepackage{calrsfs} 
\DeclareMathAlphabet{\pazocal}{OMS}{zplm}{m}{n}
\usepackage{tocloft,calc}
\usepackage{nicematrix}
\usepackage{tikz-cd}
\usepackage{tikz}
\usepackage{pgfplots}
\usepackage{tikz-3dplot}
\tdplotsetmaincoords{60}{115}
\usepackage{comment}

\usepackage[LGR,T1]{fontenc} 
\usepackage[utf8]{inputenc} 
\usepackage{lmodern} 
\usepackage[greek,english]{babel} 
\usepackage{alphabeta} 
\usepackage{biblatex}
\usepackage{csquotes}
\usepackage{cleveref}

\newtheorem{theor}{Theorem}[section]
\newtheorem{lem}[theor]{Lemma}
\newtheorem{prop}[theor]{Proposition}
\newtheorem{cor}[theor]{Corollary}
\newtheorem{defin}[theor]{Definition}
\newtheorem{rem}[theor]{Remark}

\def\C{\mathbb{C}}

\def\Z{\mathbb{Z}}
\def\N{\mathbb{N}}
\def\R{\mathbb{R}}

\def\F{\mathcal{F}}

\def\P{\mathbb{P}}

\DeclareMathOperator{\Prr}{P}

\DeclareMathOperator{\Span}{span}

\DeclareMathOperator{\Trun}{\downarrow}

\DeclareMathOperator{\clos}{Clos}

\newcommand{\ignore}[1]{}

\newcommand{\lp}{\langle}
\newcommand{\rp}{\rangle}

\newcommand{\ip}[2]{{\langle#1,#2\rangle}}
\DeclarePairedDelimiter{\norm}{\lVert}{\rVert}

\newcommand{\Vc}{{\mathcal{V}}}
\newcommand{\bw}{{\textbf{w}}}
\newcommand{\bn}{{\textbf{n}}}
\newcommand{\bg}{{\textbf{g}}}
\newcommand{\bh}{{\textbf{h}}}
\newcommand{\bq}{{\textbf{q}}}

\newcommand{\eps}{{\epsilon}}
\newcommand{\Bc}{{\Gamma}}

\newcommand{\Fc}{{\mathcal{F}}}

\newcommand{\wb}{{\mathbf w}}

\newcommand{\dd}{{\mathbf d}}
\newcommand{\deltaghq}{{\mathbf \Delta}_{\bg,\bh}^{\bq}}
\newcommand{\Aghq}{{\mathbf A}_{\bg,\bh}^{\bq}}
\newcommand{\an}{{ t}}
\newcommand{\ddn}{{\gamma}}
\newcommand{\rro}{{\wedge}}

\DeclareMathOperator{\dis}{\dd}

\newcommand{\Sbn}{{\mathbb{S}}_n}
\newcommand{\A}{\F[X_1,X_2,\ldots,X_n]^{\Sbn}}
\newcommand{\Sn}{{\mathcal S}_n}

\title{G-Invariant Representations using Coorbits: Injectivity Properties}
\author{{\bf Radu Balan, Efstratios Tsoukanis} \\
Department of Mathematics \\ University of Maryland \\ College Park, MD 20742 \\ email: rvbalan@umd.edu efstratios.tsoukanis@cgu.edu}

\date{\today}

\pgfplotsset{compat=1.18} 
\begin{document}

\maketitle

\begin{abstract}
Consider a finite-dimensional real vector space equipped with a finite group acting unitarily on it. We address the general problem of constructing Euclidean stable embeddings of the quotient space of orbits. Our approach is based on subsets of sorted coorbits with respect to selected window vectors. We derive conditions under which such embeddings are injective, using techniques from semi-algebraic analysis. These results are then applied to the specific case of planar rotation invariance.

\end{abstract}

\section{Introduction}

Invariant theory is the field of mathematics that studies objects that remain unchanged under certain transformations. At its core, invariant theory seeks to identify and understand the algebraic and geometric properties of mathematical objects that are preserved under transformations induced by symmetries.

 It comes as no surprise that invariant theory has applications in modern mathematics. Machine learning techniques yield impressive results when training is provided with large datasets. In some scenarios, our training dataset may be limited in size, but we are aware of certain underlying symmetries of our model structure. For instance, in graph machine learning, 
 each graph is represented by an adjacency (or weight) matrix and potentially a data matrix with same number of rows as number of nodes. For classification or regression tasks at the graph level, the output of the classification or regression function should not change if nodes are relabeled.

One approach to these kinds of problems is full data augmentation, 
which involves expanding the training set to contain the complete orbit generated by the group action on each data point. 
However, a significant challenge emerges because of the computational demands of deriving such highly symmetric functions.

Another solution is to embed the input data into an Euclidean space $\mathbb{R}^m$ using a symmetry-invariant embedding $\Psi$ and utilize $\mathbb{R}^m$ as a feature space for further processing. 
To be practical, such embeddings should separate data orbits while remaining computationally efficient. 
This problem is an instance of \emph{invariant machine learning}. Some papers deal with theoretical problems associated with invariant machine learning
\cite{mixon2022injectivity,aslan2023group,dym2022low,DUFRESNE20091979,cahill2023bilipschitz,yarotsky2021universal,edidin2025orbit}. Other articles deal with practical issues related to network implementation and optimization \cite{bronstein,deepsets,Gilmer_2017arXiv170401212G,pmlr-v97-maron19a}.

Common group actions in invariant machine learning theory are the group of permutations
\cite{sannai2020universal,cahill22,balan2022permutation},
the group of reflections \cite{mixon2022max}, and  the group of translations \cite{Cahill19}.

A related, yet slightly different problem, is the case 
of equivariant embeddings \cite{lipman2022,maron2018invariant,villar2021scalars,villar2022dimensionless,blum2022equivariant,keriven2019universal} where the task is to produce a map that intertwines between two representations of the same group.

Our work extends the approach first introduces in \cite{balan2022permutation}, and combines that work with   the {\em max filter} theory described in \cite{cahill22,mixon2022injectivity}. This paper presents a unified framework for the two approaches: the permutation invariant representation in \cite{balan2022permutation} encodes the full sorted coorbit by removing certain duplicates, while the max filter approach extracts only the maximum element of each coorbit. Our framework replaces these two extreme cases (full coorbit versus only the top element) with an embedding that selects a fixed subset of sorted coorbits. When the subset contains only the top elements, 
then the embedding turns into the max filter embedding 
 \cite{cahill22}; when the subset contains all entire coorbits, then the embedding extends the approach in \cite{balan2022permutation}.

A different approach to invariant machine learning is proposed in \cite{yarotsky2021universal}. The author of that paper constructs efficient approximations of polynomial group invariants using tools from machine learning. Our methods are completely different, and are based on the theory of coorbit representations. 

The construction of certain permutation invariant embeddings is closely connected to the phase retrieval problem, a well-studied topic in literature \cite{balan06,bcmn,CSV12,CEHV13,bendory2025generalized,edidin2023generic,alaifari2022phase}. 
In \cite{Balan2023Rel} we showed an isometric 
equivalence between the real case of the phase retrieval problem, and the permutation invariant representation for the special group $S_2$. 

Methods from semi-algebraic analysis are used to prove our main results. Our approach has been inspired by the semi-agebraic methods in \cite{dym2022low} as well as from an earlier paper \cite{balan06}.

Our paper is organized as follows:
Section \ref{sec2} introduces our embedding scheme and states the main results. 
Proofs of the main results are presented in section \ref{sec3}. Section \ref{sec4} 
applies our results to the case of planar rotations. 
A summary of results from semi-algebraic geometry, as well as proofs of certain lemmas, are included in Appendix.

\section{Notations and Statement of Main Results\label{sec2}}
Let $(\Vc,\ip{\cdot}{\cdot} )$ be a $d$-dimensional real vector space with scalar product, where $d\geq2$.
Assume $(G, \cdot)$ is a finite group of order $|G|=N$ acting unitarily on $\Vc$.
For every $g \in G$, we denote by $U_gx$ the group action on vector $x\in\Vc$.
Let $\hat{\Vc}=\Vc/\sim$ denote the quotient space with respect to the action of group $G$. 
 We denote by $[x]$ the orbit of vector $x$, i.e. $[x]= \{U_gx :g \in G\}$. The natural metric, $\dis: \hat{\Vc} \times \hat{\Vc} \to \R$, is defined by
\begin{equation}    
\dis([x],[y]) := \min_{h_1,h_2 \in G} \norm{U_{h_1}x-U_{h_2}y} = \min_{g \in G} \norm{x-U_gy}.
\end{equation}
Note $(\hat{\Vc},\dis)$ is a complete metric space.

Our goal is to construct a bi-Lipschitz Euclidean embedding of the metric space $(\hat{\Vc},\dis)$ into an Euclidean space $\R^m$.

Specifically, we want to construct a function $\Psi: \Vc \to \R^m$ such that
\begin{enumerate}
	\item  $\Psi(U_gx)=\Psi(x),\ \forall x \in \Vc,\ \forall g \in G$,\label{property1}
	\item  If $x,y \in \Vc$ are such that $\Psi(x)=\Psi(y)$, then there exist $g \in G$ such that $y=U_gx$,\label{property2}
	\item  There are $0<a<b<\infty$ such that for any $x,y \in \Vc$
	\[a \dis([x],[y])^2 \le\norm{\Psi(x)-\Psi(y)}^2\le b (\dis([x],[y]))^2.\]\label{property3}
\end{enumerate}   

The invariance property \eqref{property1} allows us to lift $\Psi$ to a map $\hat{\Psi}$ acting on the quotient space $\hat{\Vc}=\Vc/\sim$:
\[ \hat{\Psi}:\hat{\Vc}\rightarrow\R^m,\quad \hat{\Psi}([x])=\Psi(x), \quad\forall [x]\in\hat{\Vc}. \]
If a $G$-invariant map $\Psi$ satisfies property \eqref{property2} we say that $\Psi$ {\em separates} $G$-orbits in $\Vc$.

Our construction for the embedding $\Psi$ is based on the non-linear sorting map $\Trun$ described next.

\begin{defin}
	$\Trun:\R^r \to \R^r$ denoted the operator that takes as input a vector in $\R^r$ and returns the monotonically decreasing sorted  vector of same length $r$ that has same entries as the input vector:
 \[ x\in\R^r \mapsto \Trun x = (x_{\sigma(i)})_{1\leq i\leq r}~~,~~x_{\sigma(1)}\geq\cdots\geq x_{\sigma(r)} \]
 for some permutation $\sigma:\{1,\ldots,r\}\rightarrow\{1,\ldots,r\}$.
\end{defin}

For an integer $p \in \N$, we denote $[p]=\{1,2,\ldots,p\}$. 

For a set $S$, $|S|$ denotes its cardinal. Fix a $p$-tuple of vectors $\bw= (w_1,\dots,w_p)\in \Vc^p$. For any  $i \in [p]$ and $j \in [N]$ we define the operator $\Phi_{w_i,j}: \Vc \to \R$ so that $\Phi_{w_i,j}(x)$ is the $j$-th coordinate of the sorted vector $\Trun(\ip{U_gw_i}{x})_{g\in G }$.
Fix a set $S \subset [N] \times [p]$ such that $|S|=m$, and for $i \in [p]$, let $S_i=\{k \in [N]: (k,i)\in S\}$ (the $i^{th}$ column of $S$). Denote by $m_i$ the cardinal of the set $S_i$, $m_i=|S_i|$. Thus $m= \sum_{i=1}^p m_i$.  
\begin{defin}\label{sortedcoorbit}
    The sorted coorbit embedding $\Phi_{\bw,S}$ associated to windows $\bw\in\Vc^p$ and index set $S\subset [N] \times [p]$ is given by the map
\begin{equation}
    \label{emb}
\Phi_{\bw,S}:\Vc\rightarrow\R^m~,~
\Phi_{\bw,S}(x):=[\{\Phi_{w_1,j}(x)\}_{j \in S_1},\dots,\{\Phi_{w_p,j}(x)\}_{j \in S_p}] \in \R^m.
\end{equation}
For $S=[N]\times[p]$ we simply use $\Phi_{\bw}=\Phi_{\bw,[N]\times[p]}$ to denote full sorted coorbits.  
For a single window $w \in \Vc$, the full sorted coorbit embedding $\Phi_w$ is defined as:
    \[
    \Phi_w: \mathcal{V} \to \mathbb{R}^N, \quad \Phi_w(x) := \downarrow \left( \left\{ \langle x, U_g w \rangle \right\}_{g \in G} \right),
    \]
\end{defin}
\begin{figure}
 \centering
    \begin{tikzpicture}[]
    \draw[black, very thick] (6,0) rectangle (10,4)[label=right: {$\Vc=\R^d$}] {}; 
    \draw[black, very thick] (12,0) rectangle (15,4);
    \draw (8,1.8) node (yaxis) [below] {$[y]$};
     \draw  (8,3.3) node (yaxis) [below] {$[x]$};
     \draw  (8,-0.3) node (yaxis) [below] {$\hat{\Vc}$};
     \draw  (13.5,-0.3) node (yaxis) [below] {$\R^m$};
     \draw  (11,2.9) node (yaxis) [below] {$\Phi_{\bw,S}$};
     \draw  (11,1.9) node (yaxis) [below] {$\Phi_{\bw,S}$};
      \draw  (13.8,3.1) node (yaxis) [above] {$\Phi_{\bw,S}([x])$};
     \draw  (13,2.1) node (yaxis) [above] {$\Phi_{\bw,S}([y])$};

 \draw[red, thick, dotted, scale=1,domain=0:3,smooth]
  plot[parametric,id=parametric-example] function{6+sqrt(9-t*t),t};
 \draw[green, thick, dotted, scale=1,domain=0:4,smooth]
  plot[parametric,id=parametric-example] function{6+sqrt(16-t*t),t};
    \foreach \Point in {(14,3)  }{
    \node[green] at \Point {\textbullet};}
      \foreach \Point in {(13,2)  }{
    \node[red] at \Point {\textbullet};}
    \draw[dashed,->] (8.77,3) -- (13.8,3) [label=right: {$\Vc=\R^d$}] {};;
    \draw[dashed,->] (8.23,2) -- (12.8,2);
    \end{tikzpicture}
\caption{Proposed embedding scheme}
\label{fig1}
\end{figure}

\vspace{5mm}

Fix $\bw=(w_1,\dots,w_p)$ and take $S=\{(1,1),\dots,(1,p)\}\subset [N] \times [p]$. Recall that  $S_i=\{k \in [N]: (k,i)\in S\}$. 
In that case $\forall i \in [p]$, $S_i= \{1\}$, so $\Phi_{\bw,S}$ is the \emph{max filter}  map
$(\lp\lp w_1 ,x\rp \rp,\dots ,\lp\lp w_m ,x\rp \rp)^T$ from \cite{cahill2022group}. 

In \cite{cahill2022group}, it is shown that $2d$ vectors are enough for the construction of an injective embedding:

\begin{theor}[Lemma 12 in \cite{cahill2022group}] Consider any finite subgroup $G \le O(d)$. For a generic $\bw \in \Vc^p$ and for $S=\{(1,1),\dots,(1,p)\}$, the map $\Phi_{\bw,S}$ separates $G$-orbits in $\R^d$ provided that  $p\geq 2d$.
\end{theor}

 For fixed $n \in [N]$ we define the set of $n$-tuples of {\em distinct} group elements: 
\begin{equation}
H_n = \left\{(g_1,\ldots,g_n) \in G^n~{\rm such~that}~\forall i\neq j\in[n],g_i\neq g_j\right\}
\end{equation}
The cardinal of $H_n$ is $\frac{N!}{(N-n)!}$.
To add clarity to the notation, when dealing with multi-indexed $n$-tuples $g=(g_1,\ldots,g_n)\in H_n$, 
we shall often refer to its entries as $g(1)=g_1$, ... , $g(n)=g_n$. 
   
The following semi-algebraic sets play a crucial role in our paper. 
For fixed $n \in [N]$, $q\geq 0$, and $g,h \in H_n$, we define the following set:
\begin{equation} \label{eq:Gammaghq}
\Bc_{g,h}^{q} := \{(x,y)\in \Vc \times  \Vc: \dim(\Span(U_{g(1)}x-U_{h(1)}y,\dots,U_{g(n)}x-U_{h(n)}y)) = q\}.
\end{equation}
Notice that the set $\Bc_{g,h}^{q}$ is  semi-algebraic. It is not an algebraic set since its construction involves polynomial equalities {\em and} inequalities. 
We denote by $\gamma_{g,h}^{q}$ its semi-algebraic dimension,  
\begin{equation}
\label{eq:gamma_dim}
    \gamma_{g,h}^{q} = \dim\,\Bc_{g,h}^{q}.
\end{equation}
By convention, the semi-algebraic dimension of the empty set is $-1$. Thus $-1\leq \gamma_{g,h}^q\leq 2d$. 
The following quantities are central to our results:
\begin{equation}\label{eq:rhonq}
    \rho_{n}(q):=\max_{g,h\in H_n}\gamma_{g,h}^q = \max_{g,h\in H_n} \dim\,\Bc_{g,h}^{q}.
\end{equation}

We denote by $V_G$ the linear subspace of vectors invariant to the action of each group element, i.e. $V_G = \{x \in \Vc: U_gx=x,~\forall g \in G\}$. 
We let $d_G$ denote its dimension $d_G=\dim(V_G)$.
Let $W_G=\{(x,x) ~;~x\in V_G\}\subset V_G\times V_G\subset \Vc\times\Vc$. Notice that $W_G$ is a linear space of the same dimension as $V_G$, $\dim(W_G)=\dim(V_G)=d_G$.


\subsection{Main results}

Our goal is to examine pairs $(\bw,S)$ with $\bw\in \Vc^p$ and $S\subset [N]\times [p]$ such that $\hat{\Phi}_{\bw,S}:\hat{\Vc} \to \R^m$ is injective, where $m=|S|$. Recall the notation $S_i=\{k \in [N]: (k,i)\in S\}$ (the $i^{th}$ "column" of $S$).

\begin{theor}\label{theor1}
	Let $p \geq 2\,dim(\Vc)-d_G$. Then, for a generic $\bw \in \Vc^p$ with respect to Zariski topology, the following holds true:  for every $S\subset [N] \times [p]$ so that $S_i\neq\emptyset $,  $\forall i \in [p]$, the map $\hat{\Phi}_{\bw,S}$ is injective.  
\end{theor}

\begin{theor}\label{theorlow}
Let $M\subset\Vc$ be a semi-algebraic set of dimension $k$, $k\leq dim(\Vc)$. Assume $M$ is invariant to $U_g$, for each $g\in G$, i.e. $U_g(M)\subset M$. Let $p \geq 2k+1$.
Then, for a generic $\bw \in \Vc^p$ with respect to Zariski topology, the following holds true: for every $S\subset [N] \times [p]$ so that $S_i\neq\emptyset$,  $\forall i \in [p]$, the map $\hat{\Phi}_{\bw,S}{\vert}_{\hat{M}}:\hat{M}\rightarrow\R^{|S|}$ is injective.
\end{theor}

\begin{theor} \label{theoremmain2}

a.  Fix a positive integer $p$ and $\bn=(n_1,\ldots,n_p) \in [N]^p$ such that 

\begin{equation}
    \label{eq5modified}
\max_{
q_1\in[n_1],\ldots,q_p\in[n_p]
}\left(
min_{i\in[p]}\rho_{n_i}(q_i) -(q_1+\ldots+q_p) \right)
\leq d_G  .
\end{equation}

Then, for a generic ${\bw} \in \Vc^p$ with respect to Zariski topology, the following holds true:  
for every $S\subset [N] \times [p]$ so that $|S_i|\geq n_i$,  $\forall i \in [p]$, the map $\hat{\Phi}_{\bw,S}$ is injective.

b. In the special case $n_1=\dots=n_p=n$ choose $p$ so that
\begin{equation} 
 \label{eq:pTh2.8}
 p >  \max_{q\in [n]} \frac{1}{q}\left(\rho_n(q) -d_G-1\right)  . 
\end{equation}
	Then, for a generic ${\bw} \in \Vc^p$ with respect to the Zariski topology, and for every $S\subset [N] \times [p]$ such that $|S_i|\geq n$, $\forall i \in [p]$, the map $\hat{\Phi}_{\bw,S}$ is injective.  
\end{theor}

\begin{rem}
\Cref{theor1} can be seen as a corollary of \Cref{theoremmain2} part b. For $n=1$, $q$ must be $1$ and $\gamma^q_{g,h}=\gamma^1_{g,h}=2d$. Thus condition (\ref{eq:pTh2.8}) becomes $p>2d-d_G-1$ which means $p\geq 2d-d_G$. 
\end{rem}

\begin{rem}
    Let 
$$p_n = \max_{q\in [n]} \frac{1}{q}\left( \rho_n(q) -d_G-1\right) = \max_{q\in [n]} \frac{1}{q}\left(\max_{g,h\in H_n} \gamma^q_{g,h} -d_G-1\right). $$ 
Notice that in general $p_n$ is a non-integer real number. In \Cref{prop_rho} we establish that for every $1\leq n\leq N-1$, $p_{n+1}\leq p_n$ which is consistent with our intuition that additional measurements can only help for the map $\hat{\Phi}_{\bw,S}$ to become injective. 
    Furthermore, since $\gamma^q_{g,h}\leq 2d$, it follows $p_n\leq 2d-d_G-1$ for every $n$.
\end{rem}
An illustration of these theorems is provided in Section \ref{sec4} where we analyze the finite group of planar rotations.

\section{Derivations of the Main Results\label{sec3}}

In this section, we analyze the set $\Fc_S$ of $p$-tuples $\wb\in\Vc^p$ such that $\hat{\Phi}_{\wb,S}$ is not injective on the quotient space $\hat{\Vc}$. We identify a finite union of semi-algebraic sets that includes $\Fc_S$. In turn, each of these semi-algebraic sets is obtained as the projection of the total space of certain semi-algebraic vector bundles. A careful analysis of its semi-algebraic dimension provides the conclusions expressed in \Cref{theor1}, \Cref{theorlow} and \Cref{theoremmain2}. This "lift-and-project" proof technique was introduced in \cite{balan06} and has been used in several papers since \cite{Cahill19,cahill2022group,dym2022low,wang2019generalized,rong2021almost,conca2015algebraic}.

Recall that $V_G=\{x\in\Vc~,~U_g x=x,\forall g\in G\}$ and $W_G=\{(x,x)~,~x\in V_G\}$ are linear spaces of the same dimension $d_G$. The orthogonal complement $W_G^\perp$ is a linear subspace of $\Vc^2$ of dimension $2d-d_G$, with $d=dim(\Vc)$, which is characterized by $W_G^\perp = \{(x,y)\in\Vc\times\Vc~,~ \ip{x+y}{x}=0,\forall z\in V_G\}$.

For fixed $w \in \Vc$ and $i \in [N]$, recall that $\Phi_{w,j}(x)$ represents the $j$-th coordinate of $\Trun(\ip{U_gw}{x})_{g\in G}$. 
The maps $\Phi_{w,j}$ satisfy specific scaling and symmetry properties summarized in the following lemma whose proof we omit:
\begin{lem}\label{norm}
	For $j\in [N]$, $\lambda\geq 0$, $x,w\in\Vc$ and $z\in V_G$,
	\begin{gather}
		\Phi_{\lambda w,j}(x)  =\Phi_{w,j}(\lambda x)  =\lambda \Phi_{w,j}(x)\\
		\label{symprop}
		\Phi_{w,j}(x) =\Phi_{x,j}(w), \\
        \Phi_{w,j}(\lambda x+z) = \Phi_{w,j}(\lambda x) + \Phi_{w,j}(z). 
	\end{gather}    
\end{lem}

Let $ S_1(\Vc \times \Vc)=\{(x,y)\in \Vc \times \Vc: \norm{x}^2+\norm{y}^2=1\}$ denote the unit sphere of $\Vc^2$.   
The following lemma reduces the injectivity property to analysis on a smaller-dimensional manifold (and semi-algebraic set):
\begin{lem}\label{norm2}
    The following are equivalent:
    \begin{enumerate}
        \item $\hat{\Phi}_{\bw,S}:\hat{\Vc}\rightarrow\R^m$ is injective;
        \item $\forall (x,y) \in  W^{\bot}_G~,~\Phi_{\bw,S}(x)=\Phi_{\bw,S}(y) \Rightarrow x \sim y$.
        \item $\forall (x,y) \in  W^{\bot}_G\cap S_1(\Vc \times \Vc)~,~\Phi_{\bw,S}(x)=\Phi_{\bw,S}(y) \Rightarrow x \sim y$.
    \end{enumerate}
\end{lem}

\begin{proof}
\mbox{}

$1\Rightarrow 2$. This is obvious since 2. refers to a subset of $\Vc^2$, unlike 1.

$2\Rightarrow 3$. This is again obvious by same reason.

$3\Rightarrow 1$. Assume 3. holds true.
Let $x,y\in\Vc$ so that $\Phi_{\bw,S}(x)=\Phi_{\bw,S}(y)$. We need to show that $x\sim y$. 

Notice that  $(x,y)= (x_1,y_1)+ (z,z)$, for some unique  $(x_1,y_1) \in W_G^\perp $ and  $(z,z) \in W_G$. If $x_1=y_1=0$ then $x=y$ and the conclusion $x\sim y$ follows.

Assume now that $(x_1,y_1)\neq (0,0)$. Let $\lambda=\sqrt{\norm{x}^2+\norm{y}^2}$, 
$x_2=\frac{1}{\lambda}x_1$, and $y_2=\frac{1}{\lambda}y_1$. 
 Using Lemma \Cref{norm} $\Phi_{\bw,S}(x)=\lambda\Phi_{\bw,S}(x_2)+\Phi_{\bw,S}(z)$ and $\Phi_{\bw,S}(y)=\lambda\Phi_{\bw,S}(y_2)+\Phi_{\bw,S}(z)$. Therefore we obtain that $\Phi_{\bw,S}(x_2)=\Phi_{\bw,S}(y_2)$. By assumption 3. we obtain that $x_2\sim y_2$, that is $y_2=U_{g_0} x_2$ for some $g_0\in G$. Therefore $y=\lambda y_2+z=\lambda U_{g_0}x_2+U_{g_0}z=U_{g_0}x$ which means $x\sim y$.

\end{proof}

The unit sphere of $W_G^\perp$,  $S_1(W^{\bot}_G) =  W^{\bot}_G\cap S_1(\Vc \times \Vc)$ is algebraic, and hence a semi-algebraic set of dimension $2d-d_G-1$ (in fact, it is an algebraic manifold).

For $x,y \in \Vc$, $x\not\sim y$, and $j \in [N]$ the set of templates $w\in\Vc$ so that $\Phi_{w,j}(x)=\Phi_{w,j}(y)$ is included in the finite union of $d-1$-dimensional linear subspaces, $\bigcup_{h_1,h_2 \in G} \{U_{h_1}x- U_{h_2}y\}^\bot$. 
N

Let $S \subset [N] \times [p]$. We let $\Fc_S$ denote all $p$-tuples $\bw = (w_1, \dots, w_p)$ such that 
$\Phi_{\bw,S}$
fails to separate all possible non-equivalent vectors $x, y \in \Vc$:

\begin{align}\label{eq:defFs}
	\Fc_{S} := \big\{\bw \in \Vc^p 	&: \exists x,y \in \Vc\text{ with }x\nsim y\\
	&\text{ and }  \Phi_{w_i,j}(x) =\Phi_{w_i,j}(y), \forall (i,j)\in S \big\}. \nonumber
\end{align}
Following the terminology in \cite{dym2022low}, we refer to $\Fc_S$ as the "bad set", because it contains the set of $p$ -tuples $\bw \in \Vc^p$ that do not produce an injective embedding $\hat{\Phi}_{\bw,S}$. Our goal is to establish a sufficient condition for the "bad set" $\Fc_S$ to be included in a proper, closed Zariski subset of $\Vc^p$.

\subsection{The case of using one measurement from each coorbit: Theorem \ref{theor1}\label{subsec:1}}

Let 
\begin{equation}
    \label{eq:Gamma}
    \Bc:=\{(x,y) \in W_G^{\bot}\cap S_1(\Vc \times \Vc) : x \nsim y \} = S_1(W_G^\perp)\setminus \cup_{g\in G}\{(x,U_gx)~,~x\in\Vc\}.
\end{equation} 
Notice that $\Bc$ is a semi-algebraic set of dimension $2d-d_G-1$.

If the assumptions of \Cref{theor1} are satisfied then
\begin{align*}
	\Fc_{S} & \subset \bigcup_{(x,y) \in \Bc} \bigcup_{h_1,\dots,h_{2p} \in G}\left( \{U_{h_{1}}x-U_{h_{2}}y\}^{\bot} \times \dots \times \{U_{h_{2p-1}}x-U_{h_{2p}}y\}^{\bot}\right)\\
	&= \bigcup_{h_1,\dots h_{2p} \in G} \bigcup_{(x,y) \in \Bc} \left(\{U_{h_1}x-U_{h_2}y\}^{\bot} \times \dots \times \{U_{h_{2p-1}}x-U_{h_{2p}}y\}^{\bot}\right).
\end{align*}
For fixed $\bg,\bh \in G^p$, let

\begin{equation}\label{calF}
\Fc_{\bg,\bh}:=\bigcup_{(x,y) \in \Bc}  \{U_{g_1}x-U_{h_1}y\}^{\bot} \times \dots \times \{U_{g_{p}}x-U_{h_{p}}y\}^{\bot} .
\end{equation}



Notice that since $G$ is a finite group, in order to prove \Cref{theor1}, it is sufficient to show that, for any choice of $\bg,\bh \in G^p$, the set $\Fc_{\bg,\bh}$ has a semialgebraic dimension strictly smaller than $dp$. This fact is established as a consequence of \Cref{cor2.4}.

\subsection{Low-dimensional data embedding: Theorem \ref{theorlow}}
The Manifold hypothesis asserts that real-world data lie on a lower-dimensional manifold within the high-dimensional ambient space.  
Under the assumptions of Theorem \ref{theorlow} our model's data lie on a semi-algebraic set $M$ of dimension $k$, invariant under the action of the group $G$, i.e. $U_gx \in M,~\forall g \in G, x\in M$.
Note that this prior is highly general, as semi-algebraic sets encompass various model classes, including linear models, sparse representations, and generative models based on ReLU neural networks. Similar ideas for the phase retrieval problem has been studied in \cite{bendory2023phase,bendory2024transversality,amir2025stability}. 


Let $\tilde{M}$ denote the set
\begin{equation} \label{eq:Mtilde}
\tilde{M}=\{(x,y) \in M\times M: x \nsim y\}.
\end{equation} 
Note that $\tilde{M}$ is a semi-algebraic set of dimension at most $2k$.
Similarly to the argument in subsection \ref{subsec:1}, it suffices to prove that for any selection of $\bg,\bh \in G^p$, the set
\begin{equation}
\label{eq:14}
\mathcal{S}_{\bg,\bh}:=\bigcup_{(x,y) \in \tilde{M}}  \{U_{g_1}x-U_{h_1}y\}^{\bot} \times \dots \times \{U_{g_{p}}x-U_{h_{p}}y\}^{\bot}  
\end{equation}
has a semi-algebraic dimension strictly less than $dp$.
This fact is established as a consequence of the \Cref{cor2.5}.

\subsection{Using $n$ entries from each coorbit}

For $g,h\in H_n$, $g=(g(1),\ldots,g(n))$, $h=(h(1),\ldots,h(n))$, recall the semi-algebraic set $\Bc_{g,h}^{q}$ introduced in \Cref{eq:Gammaghq}, 
\[ 
\Bc_{g,h}^{q} := \{(x,y)\in \Vc \times  \Vc: \dim(\Span(U_{g(1)}x-U_{h(1)}y,\dots,U_{g(n)}x-U_{h(n)}y)) = q\}.
\]
It has semi-algebraic dimension denoted by $\gamma_{g,h}^{q}$:  $\gamma_{g,h}^{q} = dim\,\Bc_{g,h}^{q}$.

\begin{defin}
    For fixed $q \in [n]$ and $g,h \in H_n$, let   $\Delta_{g,h}^{q}$ denote the set of all pairs $(x,y) \in \Vc \times \Vc$ such that
    \begin{enumerate}
        \item $\dim(\Span\{U_{g(1)}x-U_{h(1)}y,\dots,U_{g(n)}x-U_{h(n)}y\}) = q$,
        \item $\ip{x+y}{z}=0,~\forall z \in V_G$,
        \item $\norm{x}^2+\norm{y}^2=1$.
    \end{enumerate}
\end{defin}
Explicitly,  
\begin{equation}
\Delta_{g,h}^{q}:=\Bc_{g,h}^{q}\cap W_G^\perp \cap S_1(\Vc \times \Vc)=\Bc_{g,h}^{q}\cap S_1(W_G^\perp),
\end{equation}
where $S_1(X)$ denotes the unit Euclidean sphere of the linear subspace $X$.

Notice that $\Delta_{g,h}^{q}$ is a semi-algebraic set since it is defined by a set of algebraic equalities (conditions 2 and 3 above) and inequalities (condition 1).
\begin{prop}\label{dimDelta}
Fix $n,q$ with $1\leq q\leq n \leq N$, 
and $g,h \in H_n$. If $\Bc_{g,h}^{q}$ is not empty, then $\Delta^q_{g,h}$ is a semi-algebraic set of semi-algebraic dimension given by
\[ 
dim\,\Delta_{g,h}^{q} =  \gamma_{g,h}^{q}-d_G-1. 
\]
\end{prop}

\begin{proof}

Let $(x,y) \in \Bc_{g,h}^{q}$. Then, for all $z \in V_G$ and for all $j \in [n]$, we observe that:
\[
U_{g_j}(x+z) - U_{h_j}(y+z) = U_{g_j}(x) - U_{h_j}(y).
\]
This implies that $(x+z, y+z) \in \Bc_{g,h}^{q}$ for all $z \in V_G$. Thus, $\Bc_{g,h}^{q}+W_G=\Bc_{g,h}^{q}$.

Let $\tilde{\Bc}_{g,h}^{q}=\Bc_{g,h}^{q}\cap W_G^\perp$.  
Notice that $\tilde{\Bc}_{g,h}^{q}$ is a semi-algebraic set since it is defined by a set of algebraic equalities and inequalities.
We denote by $\tilde{\gamma}_{g,h}^{q}$ its semi-algebraic dimension, $\tilde{\gamma}_{g,h}^{q} = \dim\, \tilde{\Bc}_{g,h}^{q}$.
By invariance of $\Bc_{g,h}^{q}$ with respect to addition with vectors from $W_G$, we obtain $\tilde{\Bc}_{g,h}^{q}=\Pi(\Bc_{g,h}^{q})$, where $\Pi:\Vc^2\rightarrow\Vc^2$ is the orthogonal projection onto $W_G^{\perp}$. 
Furthermore, we obtain
\[ \Bc_{g,h}^{q} = \tilde{\Bc}_{g,h}^{q}\oplus W_G
\sim \tilde{\Bc}_{g,h}^{q} \times W_G
\]
where $\sim$ denotes a semi-algebraic homeomorphism.

By Proposition \ref{cartesian} it follows that:
\begin{equation}\label{eq:1111}
\gamma_{g,h}^{q} = dim\,\Bc_{g,h}^{q} = \dim\, \tilde{\Bc}_{g,h}^{q} +dim(W_G) = \tilde{\gamma}_{g,h}^{q} + d_G.
\end{equation}

Since $\tilde{\Bc}_{g,h}^{q}$ is positively homogeneous, it can be represented in a unique way as:
\[
\tilde{\Bc}_{g,h}^{q} = \{ \lambda x \mid x \in \Delta_{g,h}^{q}, \lambda > 0 \}.
\]

Hence, $\tilde{\Bc}_{g,h}^{q}$ is semi-algebraically homeomorphic to \( \Delta_{g,h}^{q} \times (0, \infty) \) and by Proposition \ref{cartesian}, we have:
\[ \tilde{\gamma}_{g,h}^{q} = 
\dim\, \tilde{\Bc}_{g,h}^{q}  = \dim\,\Delta_{g,h}^{q} + \dim\, (0, \infty) = \dim\,\Delta_{g,h}^{q} + 1.
\]
Together with \ref{eq:1111} this concludes the proof of this Proposition.
\end{proof}

Recall the set of "bad" templates
\begin{gather}\label{eq:badset}
    \begin{aligned}
        \Fc_S:= \{\bw \in \Vc^p & : \exists (x,y) \in \Bc \text{ such that } \Phi_{w_i,j}(x) =\Phi_{w_i,j}(y)\ \forall (i,j) \in S\},
    \end{aligned}
\end{gather}
where $\Bc$ has been defined in (\ref{eq:Gamma}), 
$\Bc=\{(x,y)\in W_G^\perp \cap S_1(\Vc\times\Vc)~:~x\not\sim y\}$.

To establish the proof of \Cref{theoremmain2}, it suffices to demonstrate that for every $S$ satisfying the assumptions of the theorem, the set $\Fc_S$ is included in a finite union of closed Zariski sets of nonempty complements.

\begin{align*}
    \begin{split}
        \Fc_S &= \{\bw \in \Vc^p: \exists x,y \in \Vc,~x \nsim y : \Phi_{w_i,j}(x)=\Phi_{w_i,j}(y),~\forall (i,j) \in S\}    \\
        & \stackrel{\Cref{norm2}}{=} \{\bw \in \Vc^p: \exists (x,y) \in \Bc : \Phi_{w_i,j}(x)=\Phi_{w_i,j}(y),~\forall (i,j) \in S\}    \\
        &\subset \bigcup_{(x,y) \in \Bc} \ \bigcup_{\substack{\pi_1,\dots,\pi_p \in S_N\\ \sigma_1,\dots,\sigma_p \in S_N}} \{\bw \in \Vc^p: \ip{x}{U_{g_{\pi_i(j)}}^{-1}w_i}=\ip{y}{U_{g_{\sigma_i(j)}}^{-1}w_i},~ \forall (i,j)\in S\}    \\ & =\bigcup_{\substack{\pi_1,\dots,\pi_p \in S_N\\ \sigma_1,\dots,\sigma_p \in S_N}} 
    \bigcup_{(x,y) \in \Bc} \bigtimes_{i=1}^p  \{w \in \Vc: \ip{{U_{g_{\pi_i(j)}}x-U_{g_{\sigma_i(j)}}}y}{w}=0,~ \forall j\in S_i\} 
        \\ & =\bigcup_{\bg,\bh \in (H_n)^p} \bigcup_{(x,y) \in \Bc} \bigtimes_{i=1}^p \big(  \bigcap_{j\in S_i} \{ U_{g^i(j)}x-U_{h^i(j)}y\}^{\bot}\big) =
         \bigcup_{\bg,\bh \in (H_n)^p} \Fc_{\substack{\bg,\bh}}
    \end{split}
\end{align*}
where the last set is defined, for each $\bg=(g^1,\dots,g^p), \bh=(h^1,\dots,h^p) \in H_{n_1} \times \dots \times H_{n_p}$,
\begin{equation*}
    \begin{split}
    	\Fc_{\substack{\bg,\bh}}
            &= \bigcup_{(x,y) \in \Bc} \big(\bigtimes_{i=1}^p \bigcap_{j\in S_i}\{U_{g^i(j)}x-U_{h^i(j)}y\}^{\bot}\big)
= \\
 &= \bigcup_{(x,y) \in \Bc} \left( 
\left(
\bigcap_{j\in S_1}\{U_{g^1(j)}x-U_{h^1(j)}y\}^{\bot} 
\right) \times \cdots \times 
\left(
\bigcap_{j\in S_p}\{U_{g^p(j)}x-U_{h^p(j)}y\}^{\bot} 
\right)
\right)
\\
    \end{split}
\end{equation*}


Fix $\bg= (g^1,\dots,g^p), \bh= (h^1,\dots,h^p)$  elements of $H_{n_1} \times \dots \times H_{n_p}$. 
Each $\bg,\bh$ induces a cover of $\Bc$,
\[\Bc \subset \bigcup_{\substack{q_1,\dots,q_p \\ q_i \in [n_i]}}  \bigcap_{i=1}^p \Delta^{q_i}_{g^i,h^i}.\]

Therefore,
\begin{align*}
       	\Fc_{\substack{\bg,\bh}}
            \subset \bigcup_{\substack{q_1,\dots,q_p \\ q_i \in [n_i]}} \bigcup_{(x,y) \in \bigcap_{i=1}^p \Delta^{q_i}_{g^i,h^i}}\big(\bigtimes_{i=1}^p \Span(U_{g^i(1)}x-U_{h^i(1)}y,\dots,U_{g^i(n_i)}x-U_{h^i(n_i)}y)^{\bot}\big).
\end{align*}

The first union can be restricted only to those
$\bq=(q_1,\dots,q_p) \in [n]^p$ such that $\cap_{i=1}^p \Delta^{q_i}_{g^i,h^i}$ is not empty.
 Let $\deltaghq:=\cap_{i=1}^p \Delta^{q_i}_{g^i,h^i}$ denote this nonempty semi-algebraic set and denote:
\begin{equation}\label{eq:Fqgh}
       	\Fc_{\substack{\bg,\bh}}^{\bq}
            :=  \bigcup_{(x,y) \in \deltaghq}\big(\bigtimes_{i=1}^p \Span(U_{g^i(1)}x-U_{h^i(1)}y,\dots,U_{g^i(n_i)}x-U_{h^i(n_i)}y)^{\bot}\big).
\end{equation}

We obtained that:
\begin{equation}\label{badset}
\Fc_S \subset  \bigcup_{\bg,\bh \in H_{n_1} \times \dots \times H_{n_p}}
\bigcup_{\bq\in[n_1] \times \dots \times [n_p] : \deltaghq\neq\emptyset}
\Fc_{\substack{\bg,\bh}}^{\bq}.
\end{equation}
Our goal is to show that if $p$ satisfies \Cref{eq:pTh2.8} then each $\Fc_{\substack{\bg,\bh}}^{\bq}$ is a semi-algebraic set in $\Vc^p$ of dimension strictly smaller than $pd=dim(\Vc^p)$. This guarantees that $\Fc_S$ is contained in a Zariski closed set with nonempty complement, which in turn proves \Cref{theoremmain2}.

Let $\Aghq \subseteq \mathcal{V}^{p+2}$ be the semi-algebraic set defined by:
\[
\Aghq := \{(x, y, w_1, \ldots, w_p) \in \Vc^{p+2} : (x, y) \in \deltaghq, \, S_1 w_1 = 0, \dots, S_p w_p = 0\},
\]
where
\begin{equation} \label{eq:Si}
S_i := \sum_{k=1}^{n_i} (U_{g^i(k)}x - U_{h^i(k)}y)(U_{g^i(k)}x - U_{h^i(k)}y)^T~,~i\in[p].
\end{equation}

Let $\Pi_1:\Aghq\rightarrow\Vc^2$ denote the projection onto the first two components, $\Pi_1(x,y,w_1,\ldots,w_p)=(x,y)$, and $\Pi_2:\Aghq\rightarrow\Vc^p$ the projection onto the last two components $p$, $\Pi_2(x,y,w_1,\ldots,w_p)=(w_1,\ldots,w_p)$.
 
Notice that $\mathcal{F}_{\substack{\bg, \bh}}^{\bq}=\Pi_2(\Aghq)$.

\begin{lem}\label{lem:semivb}
$(\Aghq,\Pi_1,\deltaghq)$ is a semi-algebraic vector bundle with fibers of dimension $pd-(q_1+\ldots+q_p)$ and semi-algebraic base set $\deltaghq$ of semi-algebraic dimension at most $min_{i\in[p]}\gamma^{q_i}_{g^i,h^i}-d_G-1$ . 
\end{lem}

\begin{proof}
    
First, intersection of semi-algebraic sets is semi-algebraic. Hence $\deltaghq:=\cap_{i=1}^p \Delta^{q_i}_{g^i,h^i}$ is semi-algebraic. Its semi-algebraic dimension is upper bounded by semi-algebraic dimension of each of the sets involved in the intersection, hence by their minimum. The estimate $min_{i\in[p]}\gamma^{q_1}_{g^i,h^i}-d_G-1$ follows from \Cref{dimDelta}.

At every $(x,y)\in\deltaghq$ and for each $i\in[p]$,  the symmetric matrix $S_i = \sum_{k=1}^{n_i} (U_{g^i(k)}x - U_{h^i(k)}y)(U_{g^i(k)}x - U_{h^i(k)}y)^T$ introduced in (\ref{eq:Si}) has rank $q_i$. Thus, at each $(x,y)\in\deltaghq$, the set of $(w_1,\ldots,w_p)\in\Vc^p$ such that ($x,y,w_1,\ldots,w_p)\in\Aghq$, forms a vector space of dimension $pd-(q_1+\cdots+q_p)$.

Let $G_{pd,k}=\{A\in\R^{pd\times pd}\,\vert\, A^T=A\,,\,A^2=A\,,\,trace(A)=k\}$ 
denote the Grassmannian manifold of orthogonal projections in $\Vc^p\simeq \R^{pd}$ of rank $k=q_1+\ldots+q_p$.
Consider now the mapping $f:\deltaghq\rightarrow G_{pd,k}(\R)$, $(x,y)\mapsto f(x,y)=P$ defined explicitly by $P=\bigoplus_{i=1}^p 
S_i^{\dagger}S_i$,
where $S_i^{\dagger}$ denotes the pseudo-inverse of $S_i$, 
or implicitly by $P^T=P,\,P^2=P,\,trace(P)=k$ and $P(\oplus_{i=1}^p S_i)=(\oplus_{i=1}^p S_i)P=\oplus_{i=1}^p S_i$. 
Notice that $f$ is a continuous semi-algebraic mapping and $(\Aghq,\Pi_1,\deltaghq) \simeq f^*(\xi^{\perp}_{pd,k})$, the pullback of the algebraic vector bundle $\xi^{\perp}_{pd,k}=(E_{pd,k}^\perp,\pi_{pd,k},G_{pd,k}))$, with  
$E_{pd,k}^\perp = \{(A,v)\in G_{pd,k}\times \R^{pd}\,|\,Av=0\}$ (see Proposition 12.1.4, Proposition 12.1.8, the discussion after Definition 12.7.1, and Proposition 12.7.7 in \cite{bochnak2013real}) . This establishes that $(\Aghq,\Pi_1,\deltaghq)$ is a semi-algebraic vector bundle with fibers of dimension $pd-k$. 
\end{proof}

\vspace{10mm}

\Cref{lem:semivb} admits the following two corollaries that establish \Cref{theor1} and \Cref{theorlow}.

\begin{cor}\label{cor2.4}
Assume $n=1$. Fix $\bq=(1,1,\ldots,1)\in\N^p$. For every $\bg=(g_1,\dots g_p)\in G^p$ and $\bh=(h_1,\dots,h_p)\in G^p$ the following hold true:
\begin{enumerate}
\item The semi-algebraic base set $\deltaghq$ equals $\Bc$ defined in \cref{eq:Gamma}:
\[\hspace{-5mm} \deltaghq = \Bc:=\{(x,y) \in W_G^{\bot}\cap S_1(\Vc \times \Vc) : x \nsim y \} = S_1(W_G^\perp)\setminus \cup_{g\in G}\{(x,U_gx)~,~x\in\Vc\}. \] 
\item $(\Aghq,\Pi_1, \deltaghq=\Bc)$  is a semi-algebraic vector bundle with fibers of dimension $p(d-1)$ and semi-algebraic base set $\Bc$ of semi-algebraic dimension at most $2d-d_G-1$.
\item The total space $\Aghq$ is a semi-algebraic set of semi-algebraic dimension at most $\dim\,\Aghq=2d-d_G-1+p(d-1)$.
\end{enumerate}
\end{cor}
\begin{proof}(\Cref{cor2.4})
    This corollary is a special case of \Cref{lem:semivb} in the case $n=1$ which requires $q_1=\cdots=q_p=1$ and has upper bounds $\gamma^{q_i}_{g^i,h^i}\leq 2d$.
\end{proof}

\begin{cor}\label{cor2.5}
Let $\tilde{M} \subset \Bc$ be the semi-algebraic set of dimension at most $2k$ introduced in equation (\ref{eq:Mtilde}). 
 Let
\begin{align*}
\textbf{A}^{\tilde{M}}_{\bg,\bh} = \{(x, y, w_1, \ldots, w_p) \in \Vc^{p+2} : (x, y) \in \tilde{M}, \\ \ip{U_{g_1}x-U_{h_1}}{w_1}  = 0, \dots, \ip{U_{g_{p}}x-U_{h_{p}}}{w_p} = 0\}.
\end{align*}
The following hold true:
\begin{enumerate}
\item $\tilde{M}$ is invariant under the group action, i.e., $\forall g\in G$ and $(x,y)\in\tilde{M}$, $(U_g x, U_g y)\in \tilde{M}$. 
\item $(\textbf{A}^{\tilde{M}}_{\bg,\bh},\Pi_1, \tilde{M})$  is a semi-algebraic vector bundle with fibers of dimension $p(d-1)$ and semi-algebraic base set $\tilde{M}$ of semi-algebraic dimension at most $2k$.
\item The total space $\textbf{A}^{\tilde{M}}_{\bg,\bh}$ is a semi-algebraic set of semi-algebraic dimension at most $\dim\,\textbf{A}^{\tilde{M}}_{\bg,\bh}=2k+p(d-1)$.
\end{enumerate}
\end{cor}
\begin{proof}(\Cref{cor2.5})

If $\tilde{M}=\emptyset$ then the conclusion follow trivially since semialgebraic dimension of the empty set is -1 by convention. 

Assume now that $\tilde{M}\neq\emptyset$.

Part 1. follows from the fact that $M$ is $U_g$ invariant for every $g$. 

Parts 2 and 3 follow by restricting the semi-algebraic vector bundle 
$(\Aghq,\Pi_1, \deltaghq=\Bc)$ from Corollary \ref{cor2.4}, to the semi-algebraic set $\tilde{M}\subset \Bc$. 
\end{proof}

{\bf Proof of \Cref{theor1}}

    From \Cref{cor2.4} we have that $\Aghq$ is a semi-algebraic set of dimension  at most $2d-d_G-1+ p(d-1)$. 
    From \Cref{dimproject} we have that the set
    $\Fc_{\bg,\bh}$ 
    defined in equation (\ref{calF}) has dimension at most $2d-d_G-1+ p(d-1)$. Therefore, if $p\geq 2d-d_G$ then
    $\dim(\Fc_{\bg,\bh})<pd$. This concludes the proof of \Cref{theor1}. $\Box$
    
\vspace{3mm}

{\bf Proof of \Cref{theorlow}}

From \Cref{cor2.5} the set $\textbf{A}^{\tilde{M}}_{\bg,\bh}$ is a semialgebraic set of dimension at most $2k+ p(d-1)$. 
From \Cref{dimproject} we have that the set
$\mathcal{S}_{\bg,\bh}$ 
defined in equation (\ref{eq:14}) has dimension at most $2k+ p(d-1)$. Therefore, if $p\geq 2k+1$ then $\dim(\mathcal{S}_{\bg,\bh})<pd$.
This concludes the proof of \Cref{theorlow}.
$\Box$    
\vspace{3mm}

{\bf Proof of \Cref{theoremmain2}}

Part a. 
\Cref{lem:semivb} implies that $\Aghq$ has semi-algebraic dimension at most 
\begin{equation} \label{eq:m}
m(\bg,\bh,\bq):=pd-(q_1+\ldots+q_p)+
min_{i\in[p]}\gamma^{q_i}_{g^i,h^i}-d_G-1. 
\end{equation}
In turn, from \Cref{dimproject} $\mathcal{F}_{\substack{\bg, \bh}}^{\bq}=\Pi_2(\Aghq)$ has semi-algebraic dimension upper bounded by the same $m(\bg,\bh,\bq)$. The proof of \Cref{theoremmain2} is complete if:
\[ \max_{\begin{array}{c}
g^1\in H_{n_1},\ldots,g^p\in H_{n_p} \\
h^1\in H_{n_1},\ldots,h^p\in H_{n_p} \\
q_1\in[n_1],\ldots,q_p\in[n_p]
\end{array} }
 m(\bg,\bh,\bq)<pd \]
Or, more explicitely, if
\ignore{
\[ \max_{\begin{array}{c}
g^1\in H_{n_1},\ldots,g^p\in H_{n_p} \\
h^1\in H_{n_1},\ldots,h^p\in H_{n_p} \\
q_1\in[n_1],\ldots,q_p\in[n_p]
\end{array} }
pd-(q_1+\ldots+q_p)+
min_{i\in[p]}\gamma^{q_i}_{g^i,h^i}-d_G-1  <pd \]

Or

\[ \max_{\begin{array}{c}
g^1\in H_{n_1},\ldots,g^p\in H_{n_p} \\
h^1\in H_{n_1},\ldots,h^p\in H_{n_p}
\end{array} }
\left(
\max_{
q_1\in[n_1],\ldots,q_p\in[n_p]
}\left(
min_{i\in[p]}\gamma^{q_i}_{g^i,h^i} -(q_1+\ldots+q_p) \right)\right)
\leq d_G \]
which is precisely equation(\ref{eq5new}), 

or
}
\[ 
\max_{
q_1\in[n_1],\ldots,q_p\in[n_p]
}
\left(
\max_{\begin{array}{c}
g^1\in H_{n_1},\ldots,g^p\in H_{n_p} \\
h^1\in H_{n_1},\ldots,h^p\in H_{n_p}
\end{array} }
\left(
min_{i\in[p]}\gamma^{q_i}_{g^i,h^i} \right) -(q_1+\ldots+q_p) \right)
\leq d_G. \]
For two functions $f_1:A_1\rightarrow\R$, $f_2:A_2\rightarrow\R$, with $A_1,A_2$ finite sets, $\max_{x\in A_1,y\in A_2} \min(f_1(x),f_2(y)) = \min(\max_{x\in A_1}f_1(x),\max_{y\in A_2}f_2(y))$. This result extends to $p$ functions as well, which implies:
\[\hspace{-8mm}
\max_{\begin{array}{c}
g^1\in H_{n_1},\ldots,g^p\in H_{n_p} \\
h^1\in H_{n_1},\ldots,h^p\in H_{n_p}
\end{array} }
\left(
min_{i\in[p]}\gamma^{q_i}_{g^i,h^i} \right) = 
min_{i\in[p]}
\left(
\max_{\begin{array}{c}
g^1\in H_{n_1},\ldots,g^p\in H_{n_p} \\
h^1\in H_{n_1},\ldots,h^p\in H_{n_p}
\end{array} }
\gamma^{q_i}_{g^i,h^i}
\right) = \min_{i\in[p]}\rho_{n_i}(q_i).
\]
This establishes (\ref{eq5modified}).

Part b. Now assume that $n_1=\dots=n_p=n$.
In this case we want to bound from above the following term:
\begin{align*}
\max_{\begin{array}{c}
g^1,\ldots,g^p\in H_{n}^p \\
h^1,\ldots,h^p\in H_{n}^p
\end{array} }
\left(
\max_{
q_1,\ldots,q_p\in[n]^p
}\left( pd-(q_1+\ldots+q_p)+
min_{i\in[p]}\gamma^{q_i}_{g^i,h^i}-d_G-1  \right)\right)
\end{align*}
Without loss of generality, after permuting  $q_1, \dots, q_p$,  we can assume that  $q_1 \leq \dots \leq q_p$.

Then
\begin{align*}
& 
\max_{\begin{array}{c}
g^1,\ldots,g^p\in H_{n}^p \\
h^1,\ldots,h^p\in H_{n}^p
\end{array} }
\left(
\max_{
q_1,\ldots,q_p\in[n]^p
}\left( pd-(q_1+\ldots+q_p)+
min_{i\in[p]}\gamma^{q_i}_{g^i,h^i}-d_G-1  \right)\right)  \le 
 \\ & \le 
\max_{\begin{array}{c}
g^1,\ldots,g^p\in H_{n}^p \\
h^1,\ldots,h^p\in H_{n}^p
\end{array} }
\left(
\max_{q_1\in[n]}
\left( p(d-q_1)+\gamma^{q_1}_{g_1,h_1} -d_G-1 \right)\right)  \le \\
\le & \max_{
g,h\in H_{n}}
\left(
\max_{
q\in[n]
}\left( p(d-q)+\gamma^{q}_{g,h} -d_G-1 \right)\right).
\end{align*}
On the other hand, the reverse inequality is trivial:
\begin{align*}
\max_{\begin{array}{c}
g^1,\ldots,g^p\in H_{n}^p \\
h^1,\ldots,h^p\in H_{n}^p
\end{array} }
\left(
\max_{
q_1,\ldots,q_p\in[n]^p
}\left( pd-(q_1+\ldots+q_p)+
min_{i\in[p]}\gamma^{q_i}_{g^i,h^i}-d_G-1  \right)\right) \ge \\ \ge \max_{
g,h\in H_{n}}
\left(
\max_{
q\in[n]
}\left( p(d-q)+\gamma^{q}_{g,h} -d_G-1 \right)\right) 
\end{align*}
because the right hand-side is obtained as a special case, inside the search space of the left hand-side, for $g^1=\cdots=g^p=g$, $h^1=\cdots=h^p=h$ and $q_1=\cdots=q_p=q$. Therefore
\begin{align*}
\max_{\begin{array}{c}
g^1,\ldots,g^p\in H_{n}^p \\
h^1,\ldots,h^p\in H_{n}^p
\end{array} }
\left(
\max_{
q_1,\ldots,q_p\in[n]^p
}\left( pd-(q_1+\ldots+q_p)+
\max_{i\in[p]}\gamma^{q_i}_{g^i,h^i}-d_G-1  \right)\right) = \\ = \max_{
g,h\in H_{n}}
\left(
\max_{
q\in[n]
}\left( p(d-q)+\gamma^{q}_{g,h} -d_G-1 \right)\right).
\end{align*}

Finally, the condition
\[\max_{
g,h\in H_{n}}
\left(
\max_{
q\in[n]
}\left( p(d-q)+\gamma^{q}_{g,h} -d_G-1 \right)\right) <pd  \]
is equivalent to 
\[ 
p>  \max_{q\in [n]} \frac{1}{q}\left(\max_{g,h\in H_n} \gamma_{g,h}^q -d_G-1\right).\]
$\Box$

\subsection{Additional Results}

In this subsection we derive some simple results relating the dependency of $\rho_n(q)$ on $n$ and $q$.

\begin{prop}\label{prop_rho}
Let 
$$p_n = \max_{q\in [n]} \frac{1}{q}\left( \rho_n(q) -d_G-1\right) = \max_{q\in [n]} \frac{1}{q}\left(\max_{g,h\in H_n} \gamma^q_{g,h} -d_G-1\right). $$
The following hold true:
\begin{enumerate}
    \item For every $1\leq n\leq N-1$ and $1\leq q\leq n$, 
    $$\rho_{n+1}(q)\leq \rho_n(q)\leq max(\rho_{n+1}(q),\rho_{n+1}(q+1)).$$
    \item For every $1\leq n\leq N-1$, $\rho_{n+1}(n+1)\leq \rho_n(n)$.
    \item For every $1\leq n\leq N-1$, $p_{n+1}\leq p_n$. 
    \item $p_N=\min_n p_n\leq 2d-d_G-1$.
\end{enumerate}
\end{prop}

\begin{proof}
1. For $1\leq n<N$, $1\leq q\leq n$ and $g,h\in H_n$ let $g',h'\in H_{n+1}$ be extensions of $g,h$ to $n+1$ elements, i.e., $g'{\vert}_{\{1,2,\ldots,n\}}=g$ and $h'{\vert}_{\{1,2,\ldots,n\}}=h$. Then the following  inclusion is obvious
\[ 
    \Gamma^q_{g,h} \subset \left(  \Gamma^q_{g',h'} \cup\Gamma^{q+1}_{g',h'}  \right)
\]
This establishes that $\gamma^q_{g,h}\leq max(\gamma^q_{g',h'},\gamma^{q+1}_{g',h'})$. Taking the maximum over $g,h\in H_n$ we obtain $\rho_n(q)\leq max(\rho_{n+1}(q),\rho_{n+1}(q+1))$.

Let $1\leq n<N$, $1\leq q\leq n$ and $g',h'\in H_{n+1}$. Denote by $S\subset H_n\times H_n$ the set of pairs $(g,h)$ that are restrictions of $g',h'$ to any subset of $n$ entries, i.e., $g=g'{\vert}_I$ and $h=h'{\vert}_I$ for some $I\subset [n+1]$ with $|I|=n$. 
Then the following inclusion is obvious
    \[ 
    \Gamma^q_{g',h'} \subset \cup_{(g,h)\in S} \Gamma^q_{g,h}
    \]
This establishes that $\gamma^q_{g',h'}\leq max_{(g,h)\in S} \gamma^q_{g,h}$. Taking the maximum over $g',h'\in H_{n+1}$ we obtain that $\rho_{n+1}(q)\leq \rho_{n}(q)$.
    This completes the proof of part 1.

2. Let $1\leq n<N$ and $g',h'\in H_{n+1}$. Denote by $g,h\in H_n$ the restriction of $g',h'$ to the first $n$ entries. 
Then every $(x,y)\in\Gamma^{n+1}_{g'h'}$ is also in $\Gamma^n_{g,h}$ since $\{U_{g'(1)}x-U_{h'(1)}y,\ldots,U_{g'(n+1)}x-U_{h'(n+1)}y\}$ is linearly independent. Thus $\gamma^{n+1}_{g',h'}\leq \gamma^n_{g,h}$. Taking the maximum over $g',h'$ we obtain $\rho_{n+1}(n+1)\leq \rho_n(n)$. 

3. This follows from 1 and 2 immediately.

4. Since every $\Gamma^q_{g,h}$ is a subset of $\Vc\times \Vc$ it follows $\gamma^q_{g,h}\leq 2d$ and the conclusion follows.
    
\end{proof}

\section{Application to planar rotations \label{sec4}}

In this section, we analyze the representation of the abelian group $G=\Z_N=\{0,1,\ldots,N-1\}$ identified with a subgroup of $SO(2)$, namely of planar (2D) rotations of multiples of $a:=\frac{2\pi}{N}$.


Let $d=2$, $\Vc=\R^2$, and let $U_b$ denote the counterclockwise rotation by $b$ radians, i.e, 
\[ U_{b}= 
\left[ 
\begin{array}{cc}
cos(b) & -sin(b) \\
sin(b) & cos(b)
\end{array}
\right] .\]   For a fixed $N \in \N$, let $a=\frac{2\pi}{N}$. Then $\Z_N$ is isomorphic to the group $\langle U_a \rangle= 
\{1, U_a, U_a^2, \dots, U_a^{N-1}\}$, the abelian group generated by $U_a$, with $1=U_0=U_a^0=U_a^N=U_{2\pi}$, the identity matrix.

Our task is to analyze Euclidean embeddings of $\widehat{\R^2}=\R^2/G$ produced by the sorted coorbit construction introduced in \Cref{sortedcoorbit}. First we present a geometric embedding of $\widehat{\R^2}$ using a "gluing" method. Then we show the following results.
\begin{enumerate}
   \item Any single full coorbit cannot produce an injective embedding.
   \item Any coorbit embedding in $\R^2$ using only two measurements (even from two different coorbits) cannot be injective.     
    \item An injective embedding in $\R^3$ is possible using only two windows $w_1,w_2$ and a total of three measurements from their coorbits.
   \item We compute explicitly $\rho_1(q)$ and $\rho_2(q)$ from \Cref{eq:rhonq} and show that Theorem \ref{theoremmain2} provides embeddings in $\R^4$ with only three windows.
\end{enumerate}

\subsection{A geometric Euclidean embedding of the quotient space $\widehat{\R^2}$}

First, the metric space $\widehat{\R^2}$ admits a natural 3D Euclidean embedding by identifying the x-semiaxis $[0,\infty)\times\{0\}$ with the 
slanted semi-line $\{r(cos(a),sin(a)):r\geq 0\}$. We obtain a 2D-cone (see \Cref{fig:cone_embedding}) that embeds in $\R^3$ via:
\begin{equation}
    \label{eq:cone}
\hspace{-5mm}    \Psi:\R^2\rightarrow\R^3~,~
    (x,y) \mapsto \Psi(x,y)=\left(\frac{r}{N}cos(N\theta),\frac{r}{N}sin(N\theta),r\sqrt{1-\frac{1}{N^2}}\right)
\end{equation}
 where $r=|x+iy|=\sqrt{x^2+y^2}$ and $\theta=Arg(x+iy)$.
 The distance $\dd$ of the metric space $(\widehat{\R^2},\dd)$ is given by 
 \[\dd((x_1,y_1),(x_2,y_2))=\sqrt{r_1^2+r_2^2-2r_1r_2cos(min(|\theta_1-\theta_2|,\frac{2\pi}{N}-|\theta_1-\theta_2|))},\] with $r_k=|x_k+iy_k|$, $\theta_k=Arg(x_k+iy_k)\,(mod)\frac{2\pi}{N}$, $k=1,2$. 
 The map $\Psi$ is bi-Lipschitz on $\widehat{\R^2}$ with
 \begin{equation}
     \frac{1}{N^2sin(\frac{\pi}{2N})^2} \dd((x_1,y_1),(x_2,y_2))^2 \leq \norm{\Psi(x_1,y_1)-\Psi(x_2,y_2)}_2^2 \leq \dd((x_1,y_1),(x_2,y_2))^2
 \end{equation}

\begin{figure}[ht]
    \centering
    \begin{tikzpicture}
    \begin{axis}[
        view={45}{25},
        axis lines=center,
        ticks=none,
        xlabel={$x$}, ylabel={$y$}, zlabel={$z$},
        samples=25,
        domain=0:1,
        y domain=0:360,
        z buffer=sort,
        colormap/blackwhite,
        shader=flat,
        point meta=none
    ]
    \addplot3[
        surf,
        mesh/ordering=y varies,
        draw=black,
        fill=none
    ]
    ({x*cos(y)}, {x*sin(y)}, {x});
    \end{axis}
    \end{tikzpicture}
    \caption{Visualization of the cone embedding $\Psi(x,y)$ described by $\left(\frac{r}{N}\cos(N\theta), \frac{r}{N}\sin(N\theta), r\sqrt{1 - \frac{1}{N^2}}\right)$.}
    \label{fig:cone_embedding}
\end{figure}

Note for large $N$, the distortion $\frac{b_0}{a_0}=Nsin(\frac{\pi}{2N})\rightarrow\frac{\pi}{2}\approx 1.57$. On the other hand if $G=SO(2)$, then $(x,y)\mapsto \Psi_3(x,y)=\sqrt{x^2+y^2}$ provides an isometric embedding of $\R^2/SO(2)$ onto $[0,\infty)\subset\R$.
For finite $N$, a lower-dimensional bi-Lipschitz embbedding is obtained by dropping the third coordinate in $\Psi$, that is, $(x,y)\mapsto (\frac{r}{N}cos(N\theta),\frac{r}{N}sin(N\theta))$. This provides an embedding of $\widehat{\R^2}=\R^2/\Z_N$ into $\R^2$ (a lower dimensional space), at the expense of a much larger distortion, namely $N$.

\subsection{Sorted coorbits}

Now we derive closed form expressions for sorted coorbits. 
We shall use interchangeably cartesian and polar coordinates expressed via complex numbers. Let $[w],[x]\in\widehat{\R^2}$. For these equivalence classes we choose representatives with polar angles in $[0,\frac{2\pi}{N})$: $w = |w|e^{i\phi}$, $x = |x|e^{i\theta}$, where $|w|=\norm{w}$ and $|x|=\norm{x}$ denote their Eucldean norms (or polar radia), and $\phi,\theta\in[0,\frac{2\pi}{N})$.

The sorted coorbit is described in the following result.

\begin{prop}\label{sortedcoobit}
    Let $w = |w|e^{i\phi}$, $x = |x|e^{i\theta}$, with $\phi,\theta\in[0,a)$ and $a:=\frac{2\pi}{N}$. The sorted coorbit $\Phi_w(x):=\downarrow(\ip{x}{U_a^k w})_{k\in[N]}\in\R^N$ is given by:
\begin{equation}\label{eq:phiw}
    \Phi_w(x) = |x|\,|w|\,cos(\an)e_0 + |x|\,|w|\,sin(\an) f_0
\end{equation}
where $\an = \rro_{\phi}(\theta):=min(|\theta-\phi|,a-|\theta-\phi|)$ and
\begin{eqnarray}
\hspace{-5mm} e_0 & = & \left[ 
\begin{array}{ccccccc}
1 & cos(a) & cos(a) & cos(2a) & cos(2a) & \cdots & cos(pa)
\end{array}
\right]^T  \\
\hspace{-5mm} f_0 & = & \left[ 
\begin{array}{ccccccc}
0 & sin(a) & -sin(a) & sin(2a) & -sin(2a) & \cdots & (-1)^N sin(pa)
\end{array}
\right]^T 
\end{eqnarray}
with $p=\lfloor \frac{N}{2}\rfloor$. More explicitly,
\begin{equation}\label{sortedcoobitfinal}
    (\Phi_w(x))_{2k+1} = |x|\,|w|\,cos(\an+ka)~,~
    (\Phi_w(x))_{2k+2} = |x|\,|w|\,cos((k+1)a-\an)
\end{equation}
for $k=0,1,\ldots$. More compactly,
\begin{equation}
    \label{eq:onesort}
    (\Phi_w(x))_q = |x|\,|w|\,cos(t-(-1)^q\lfloor\frac{q}{2}\rfloor a)~~,~~q=1,2,\ldots,N
\end{equation}

The two vectors $e_0$ and $f_0$ are orthogonal to each other and of same norm 
$\sqrt{\frac{N}{2}}$.

The $N$-vector $\Phi_w(x)$ belongs to the 2-dimensional vector space spanned by $e_0,f_0$. The range of map $\Phi_w$ as $\theta$ runs through $[0,a)$ and $|x|\geq 0$ is a double cover of the 2D convex cone inside the linear space $span\{e_0,f_0\}$ between semi-lines $\R_{+}e_0$ and $\R_{+}(cos(\frac{\pi}{N})e_0+sin(\frac{\pi}{N})f_0)$.  
\end{prop}

\begin{proof}

Define $\ddn = \theta - \phi \in (-a, a)$. The unsorted coorbit vector $v=(\ip{x}{U_a^k w})_{k\in[N]}$ is given by:
\[
v_k = |x||w| \cos(\ddn - ka), \quad k = 0,\ldots,N-1.
\]
Let $\an = \min(|\ddn|, a - |\ddn|)=\rro_\phi(\theta) \in [0,a/2]$.

    \textbf{Case 1:} $\ddn \in [0, a/2]$, so $\an = \ddn=\theta-\phi$. Then $\{ cos(\ddn - ka) \mid  k=0,\ldots, N-1\} = \{cos( ka - \an) \mid k = 0,\ldots,N-1\} = 
    \{cos(\beta_k)\mid k=0,1,\ldots,N-1\}$, 
    where $\beta_k = \min( ka - \an, 2\pi - ( ka-\an)) \in [0,\pi]$.
\begin{itemize}
    \item For $k\leq \frac{\pi-\an}{a}$, $\beta_k = ka - \an$.
    \item For $k>\frac{\pi-\an}{a}$, $\beta_k = 2\pi - (ka - \an)$.
\end{itemize}
Sorting $\beta_k$ in increasing order (since cosine function is monotone decreasing on $[0,\pi]$) yields the sequence $\an, a - \an, a+ \an  , 2a - \an, \ldots$.

\textbf{Case 2:} $\ddn \in [-\frac{a}{2}, 0)$, so $\an =  -\ddn$. Then $\{cos(\ddn-ka)\mid k=0,1,\ldots,N-1\} = \{cos(\an+ka)\mid k=0,\ldots,N-1\} = \{cos(\lambda_k)\mid k=0,\ldots,N-1\}$. 

  where $\lambda_k = \min(\an + ka, 2\pi - (\an + ka)) \in [0,\pi]$.
\begin{itemize}
    \item For $k\leq \frac{\pi-\an}{a}$, $\lambda_k = \an + ka$.
    \item For $k>\frac{\pi-\an}{a}$, $\lambda_k = 2\pi - (\an + ka)$.
\end{itemize}
Sorting $\lambda_k$ in increasing order  yields the sequence $\an, a - \an, a+ \an  , 2a - \an, \ldots$.

\textbf{Case 3:} $\ddn \in (a/2, a)$, so $\an = a - \ddn$. Then $\{\cos(\ddn-ka)\mid k=0,1,\ldots,N-1\} = \{\cos(a-\an-ka)\mid k=0,\ldots,N-1\} = \{\cos(ka-(a-\an))\mid k=0,\ldots,N-1\} = \{\cos(\zeta_k)\mid k=0,1,\ldots,N-1\}$,
where $\zeta_k = \min( |ka - (a-\an)|, 2\pi - |ka - (a-\an)|) \in [0,\pi]$.
\begin{itemize}
    \item This is the same set of cosine values as in Case 1, with $\an$ replaced by $a-\an$.
    \item The set of corresponding angles in $[0, \pi]$ is $\{a-\an, \an, a+\an, 2a-\an, \dots \}$, which is a permutation of the set of angles from Case 1.
\end{itemize}
Sorting the set of angles $\{\zeta_k\}$ in increasing order yields the sequence $\an, a - \an, a+ \an , 2a - \an, \ldots$.

\textbf{Case 4:} $\ddn \in (-a, -a/2)$, so $\an = a + \ddn$. Then $\{\cos(\ddn-ka)\mid k=0,1,\ldots,N-1\} = \{\cos(\an-a-ka)\mid k=0,\ldots,N-1\} = \{\cos(ka+a-\an)\mid k=0,\ldots,N-1\} = \{\cos(\xi_k)\mid k=0,1,\ldots,N-1\}$,
where $\xi = \min( |(k+1)a - \an|, 2\pi - |(k+1)a - \an|) \in [0,\pi]$.
\begin{itemize}
    \item Since $k$ runs from $0$ to $N-1$, $k+1$ runs from $1$ to $N$. The set of values $\{\cos((k+1)a - \an)\}$ is the same as $\{\cos(ka - \an)\}$ due to the $N$-periodicity in $k$ of the cosine values.
    \item The set of corresponding angles in $[0, \pi]$ is therefore identical to the set from Case 1.
\end{itemize}
Sorting the set of angles $\{\xi_k\}$ in increasing order yields the sequence $\an, a - \an, a+ \an , 2a - \an, \ldots$.

In all these cases we obtain that, sorted list
$\{cos(\ddn -ka)\mid k=0,\ldots,N-1\}$ turns into $\left(cos(\an),cos(a-\an),cos(a+\an),cos(2a-\an),\ldots
\right)$. This proves \Cref{sortedcoobitfinal} and \Cref{eq:phiw}.

Next we prove the metric properties of $e_0$ and $f_0$.

\textbf{Case 1: \(N\) even}, $N=2p$

\begin{align*}
\|\mathbf{f}_0\|^2 &= 0 + \sum_{l=1}^{p-1} 2 \sin^2(la) + 0 
= \sum_{l=1}^{p-1} 2 \cdot \frac{1 - \cos(2la)}{2} 
= \sum_{l=1}^{p-1} (1 - \cos(2la)) \\
&= \sum_{l=1}^{p-1} 1 - \sum_{l=1}^{p-1} \cos\left( \frac{4\pi l}{N} \right) 
= (p - 1) - (-1) = p = \frac{N}{2}
\end{align*}

\textbf{Case 2: \(N\) odd}, $N=2p+1$

\begin{align*}
\|\mathbf{f}_0\|^2 &= 0 + \sum_{l=1}^{p} 2 \sin^2(la) 
= \sum_{l=1}^{p} 2 \cdot \frac{1 - \cos(2la)}{2} 
= \sum_{l=1}^{p} (1 - \cos(2la)) \\
&= \sum_{l=1}^{p} 1 - \sum_{l=1}^{p} \cos\left( \frac{4\pi l}{N} \right)
= p - \left( -\frac{1}{2} \right) = p + \frac{1}{2} = \frac{N}{2}
\end{align*}

In the inner product $\bm{e}_0 \cdot \bm{f}_0 = \sum_{j=1}^{N} (\bm{e}_0)_j (\bm{f}_0)_j$ each two consecutive terms cancel each other, 2 with 3, 4 with 5, and so on. The first product also cancels. The only remaining possible term is when $N$ is even. Then the last term $j=N=2p$ satisfies $\cos(pa) \sin(pa) = \frac{1}{2}\sin(2pa) = \frac{1}{2}\sin(2\pi) = 0$.
Thus, $\bm{e}_0 \cdot \bm{f}_0 = 0$ are orthogonal.

Next, a direct computation shows $\norm{e_0}^2+\norm{f_0}^2=N$. Thus $\norm{e_0}=\norm{f_0}=\sqrt{\frac{N}{2}}$. 
These prove the metric properties of $e_0$ and $f_0$.

Finally, \Cref{eq:phiw} proves that  the image of the angular cone of angle $a$ under $\Phi_w$ is the 2D cone inside $span(e_0,f_0)$ spanned between semi-lines $\R_+ e_0$ and $\R_0(cos(\frac{a}{2}) e_0)+sin(\frac{a}{2})f_0)$. 
To show that the image is a double cover of this cone, except for the boundary lines, notice the following symmetry. Assume $\theta\in[0,a)$ is so that $t=min(|\theta-\phi|,a-|\theta-\phi|)\not\in\{0, \frac{a}{2}\}$. Then there exists a unique $\theta'\in[0,a)$ so that $\theta'\neq\theta$ and $min(|\theta'-\phi|,a-|\theta'-\phi|) = t$. 

This is shown as follows. If $0<t=|\theta-\phi|\leq\phi\leq\frac{a}{2}$ then choose $\theta'=2\phi-\theta\in[0,a)$. If $t=|\theta-\phi|\leq \phi$ but $\phi>\frac{a}{2}$  then choose either $\theta' =2\phi-\theta$ or $\theta'=\phi+t-a$, whichever satisfies $\theta'\in[0,a)$. If  $\phi<t=|\theta-\phi|\leq\frac{a}{2}$ then choose $\theta'=a+2\phi-\theta$. The remaining cases are treated similarly. Alternatively, the plot of the function $\theta\mapsto \rro_\phi(\theta):=min(|\theta-\phi|,a-|\theta-\phi|)$ depicted in Figure \ref{figmin} for $a=\frac{\pi}{10}$ and $\phi=\frac{\pi}{40}$ shows that there are exactly two distinct solutions $\theta,\theta'$ on $[0,a)$ for each value of $t\in(0,\frac{a}{2})$.

\begin{figure}
    \centering
    \includegraphics[width=1\linewidth]{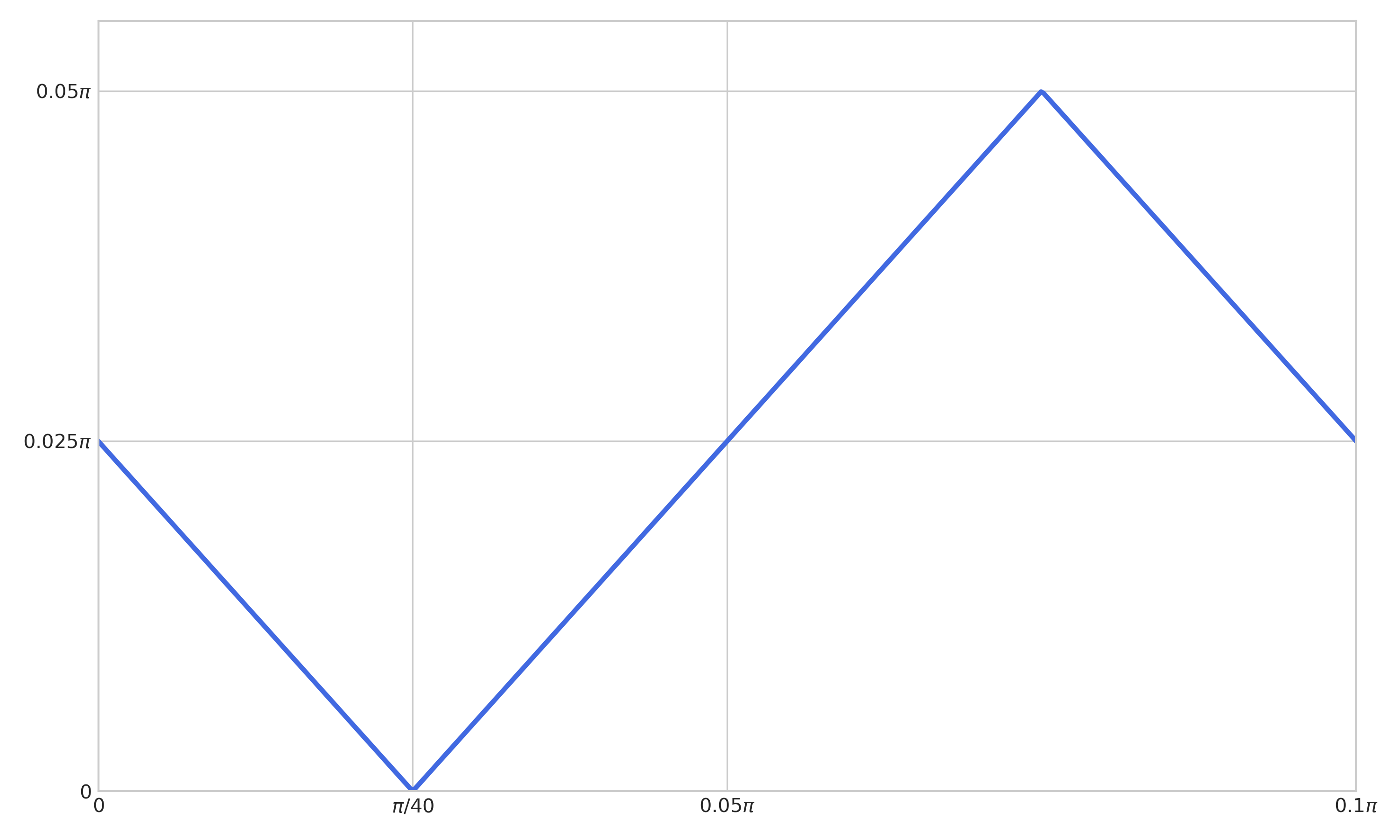}
    \caption{The plot of the function $\theta\mapsto \rro_\phi(\theta):=min(|\theta-\phi|,a-|\theta-\phi|)$ for $a=\frac{\pi}{10}$ and $\phi=\frac{\pi}{40}$.}
    \label{figmin}
\end{figure}

\end{proof}

\subsection{Injective embeddings using sorted coorbits}

\begin{prop}

Consider the sorted coorbit maps $\Phi_{w,S}$ and $\Phi_{(w_1,w_2),S}$ associated to the finite group of planar rotations $\Z_N\simeq\langle U_a \rangle =\{1,U_a,\ldots,U_a^{N-1}\}$ acting on $\Vc=\R^2$, with $N\geq 2$, $a=\frac{2\pi}{N}$, and $w,w_1,w_2\in\Vc$, $S\subset[N]$ or $S\subset[N]\times\{1,2\}$ 
\begin{enumerate}
    \item Let $S = \{(1,1), (2,1), (3,1), \dots, (N,1)\}$. For any $w\in\Vc$ the map $\hat{\Phi}_{w,S}:\hat{\Vc}\rightarrow\R^N$ is not injective.
    \item For any $w_1,w_2\in\Vc$ and $q_1,q_2\in[N]$, the map $\hat{\Phi}_{(w_1,w_2),S}:\hat{\Vc}\rightarrow\R^2$ is not injective, where  $S=\{(q_1,1),(q_2,2)\}$.
\end{enumerate}

\end{prop}

\begin{proof}
\mbox{}

1. Assume $w\neq 0$, otherwise the conclusion is immediate. 

According to \Cref{sortedcoobit} the range of the map $\Phi_w$ is a double cover of the 2D cone spanned between semilines $\R_0 e_0$ and $\R_0(cos(\frac{a}{2}e_0 + sin(\frac{a}{2})f_0)$. Hence, the quotient map cannot be injective. 

An explicit proof can be given as follows. 
Take $x = U_{a/3}w$ and $y = U_{-a/3}w$. 
Note that $x \nsim y$. On the other hand
\[ \{ \ip{x}{U_a^k w}, k\in[N] \} = \left\{ \norm{w}^2 cos(\frac{2\pi k}{N}-\frac{2\pi }{3N}),k\in[N]\right\} =\]
\[ =  \left\{ \norm{w}^2 cos(\frac{2\pi k}{N}+\frac{2\pi }{3N}),k\in[N]\right\} =\{ \ip{y}{U_a^k w}, k\in[N] \} \]
Hence $\Phi_{w,S}(x)=\Phi_{w,S}(y)$. This proves that $\hat{\Phi}_{w,S}$ is not injective.

2. 
Without loss of generality assume that $w_1$ and $w_2$ are linearly independent. Otherwise the statement is trivial.

Choose representatives $w_1$ and $w_2$ from $[w_1]$ and $[w_2]$ as in \Cref{sortedcoobit} that have associated angles $\phi_1,\phi_2\in[0,a)]$: $w_1 = \norm{w_1}U_{\phi_1}e_1$ and $w_2= \norm{w_2}U_{\phi_2}e_1$, where $e_1=(1,0)^T$. 
Our task is to construct $[x]\neq [y]$ so that $\Phi_{(w_1,w_2),S}(x)=\Phi_{(w_1,w_2),S}(y)$. 

Let $k_1=-(-1)^{q_1}\lfloor \frac{q_1}{2}\rfloor$ and $k_2=-(-1)^{q_2}\lfloor \frac{q_2}{2}\rfloor$ and denote
\[ R_{k_1,k_2}(\theta) = \frac{cos(\rro_{\phi_1}(\theta)+k_1 a)}{cos(\rro_{\phi_2}(\theta)+k_2 a)} \]

We claim that there are $\theta_x,\theta_y\in[0,a)$, with $\theta_x\neq\theta_y$, so that 
\begin{equation} \label{eq:cos}
R_{k_1,k_2}(\theta_x) = R_{k_1,k_2}(\theta_y)\neq 0
\end{equation}
with both sides finite.

Assume the claim is true. Then set
\[ 
x= \frac{1}{cos(\rro_{\phi_1}(\theta_x)+k_1a)}U_{\theta_x}e_1~~,~~
y= \frac{1}{cos(\rro_{\phi_1}(\theta_y)+k_1a)}U_{\theta_y}e_1
\]
and notice that $[x]\neq[y]$, $\Phi_{w_1,q_1}(x)=\norm{w_1}=\Phi_{w_2,q_2}(y)$ and $\Phi_{w_2,q_2}(x)=\frac{\norm{w_2}}{R_{k_1,k_2}(\theta_x)}=\Phi_{w_2,q_2}(y)$. Hence $\Phi_{(w_1,w_2),S}(x) = \Phi_{(w_1,w_2),S}(y)$ whichs proves that $\hat{\Phi}_{w,S}$ is not injective.

It remains to prove the claim stated in equation (\ref{eq:cos}).

Case 1. $q_1=q_2=q$.

Let $\theta_x=\frac{1}{2}(\phi_1+\phi_2)$ and $\theta_y= \theta_x + \eps \frac{a}{2}$, where
$ \eps=+1$ if $\theta_x\in[0,\frac{a}{2})$, and $\eps = -1$ if $\theta_x\in[\frac{a}{2},a)$.  See figure \ref{figmin1} for an illustration of these angles.

Notice that $|\theta_x - \phi_1|=\frac{1}{2}|\phi_2 - \phi_1|=|\theta_x - \phi_2|$. Hence $\rro_{\phi_1}(\theta_x) = \rro_{\phi_2}(\theta_x)$. 
Then notice that $\rro_\phi(\theta)+\rro_\phi(\theta+\frac{a}{2}) = \frac{a}{2}$ for all $\theta\in[0,\frac{a}{2})$. This shows that, additionally, $\rro_{\phi_1}(\theta_y)=\rro_{\phi_2}(\theta_y)$.
In particular $(\theta_x,\theta_y)$ satisfy \Cref{eq:cos} with $R_{k_1,k_2}(\theta_x)=R_{k_1,k_2}(\theta_y)=1$.

\begin{figure}
    \centering
    \includegraphics[width=1\linewidth]{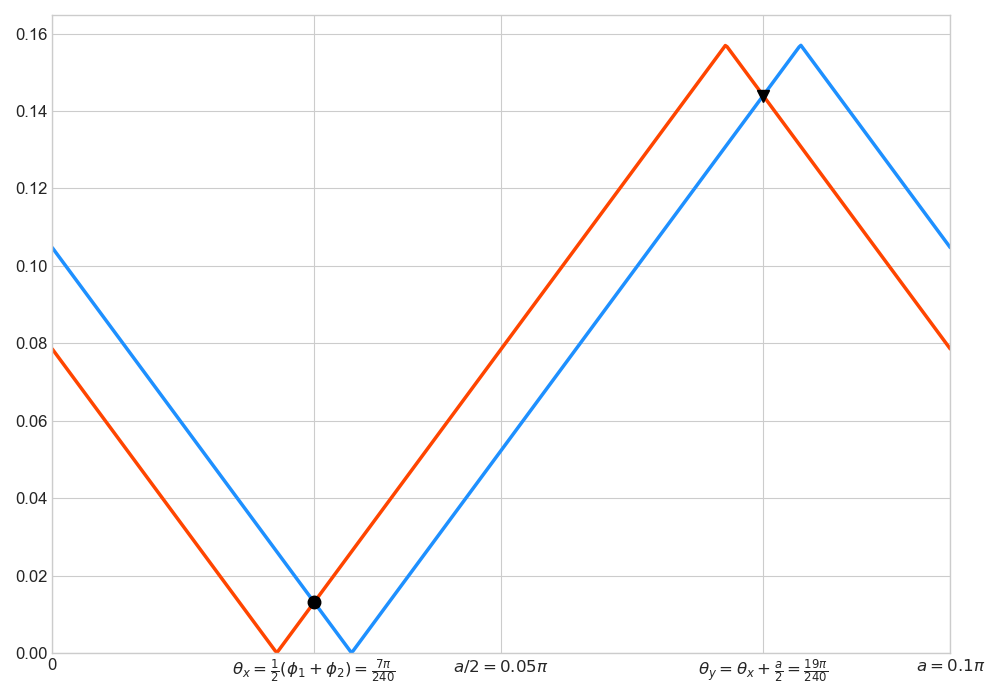}
    \caption{The plot of two functions $\rro_{\phi_1}$ and $\rro_{\phi_2}$, for $a=\frac{\pi}{10}$, $\phi_1=\frac{\pi}{30}$ and $\phi_2=\frac{\pi}{40}$, where $\rro_{\phi}(\theta):=min(|\theta-\phi|,a-|\theta-\phi|)$. Notice the two plots intersect at $\theta_x=\frac{1}{2}(\phi_1 + \phi_2)=\frac{7\pi}{240}$ and $\theta_y = \theta_x+\frac{a}{2} = \frac{19\pi}{240}$.}
    \label{figmin1}
\end{figure}

Case 2. $q_1\neq q_2$.

Notice that  $R_{k_1,k_2}(\theta)=R_{k_1,k_2}(\theta+a)$, that is, $R_{k_1,k_2}(\theta)$ is $a$-periodic. If $k_2$ is so that $\cos(\rro_{\phi_2}(\theta)+k_2 a)$ does not vanish on $[0,a]$, then the $a$-periodic function $R_{k_1,k_2}(\theta)$ is continuous. Thus the claim follows.

Similarly, if $k_1$ is so that  $\cos(\rro_{\phi_1}(\theta)+k_1 a)$ does not vanish on $[0,a]$, then the $a$-periodic function $\frac{1}{R_{k_1,k_2}(\theta)}$ is continuous and not constant. Thus the claim follows.

Lastly, if the numerator and denominator have a zero on $[0,a]$, then this only happens if  
$I_1=[k_1a, k_1a + \frac{a}{2}]$ and $I_2=[k_1a, k_1 a + \frac{a}{2}]$ contain
$ \frac{\pi}{2}$ or $-\frac{\pi}{2}$. Since $k_1\neq k_2$, $\frac{\pi}{2}$ is in one interval, and $-\frac{\pi}{2}$ should be in the other interval. 
This condition cannot be satisfied if $N=4r+1$ or $N=4r+3$. In the case $N=4r$, $k_1,k_2$  must satisfy $k_1=-k_2=\pm \frac{N}{4}$. Then $k_1a = -k_2a = \pm \frac{\pi}{2}$ and 
$R_{k_1,k_2}(\theta) = -\frac{sin(\rro_{\phi_1}(\theta))}{sin(\rro_{\phi_2}(\theta))}$.

In the case $N=4r+2$, $k_1,k_2\in\{ -\frac{N+2}{4},\frac{N-2}{4} \}$. Then $k_1a,k_2a\in\{-\frac{\pi}{2}-\frac{\pi}{N},\frac{\pi}{2}-\frac{\pi}{N}\}$. Using the fact that $\rro_{\phi}(\theta)+\rro_{\phi}(\theta+\frac{a}{2})=\frac{a}{2}$ (see \Cref{figmin}) it follows
that $R_{k_1,k_2}(\theta) = -\frac{sin(\rro_{\phi_1}(\theta+\frac{a}{2}))}{sin(\rro_{\phi_2}(\theta+\frac{a}{2}))}$.
Both sub-cases $N=4r$ and $N=4r+2$ are now solved by proving that $\theta\mapsto R(\theta)=\frac{sin(\rro_{\phi_1}(\theta))}{sin(\rro_{\phi_2}(\theta))}$ is not injective. A typical case is depicted in \Cref{figmin2}. 
Note that $R(\phi_1)=0$ and $R(\phi_1-\eps)R(\phi_1+\eps)>0$ for $\eps>0$ small enough. Since $R$ is continuous in a neighborhood of $\phi_1$, by intermediary values theorem it follows that there are $\theta_x<\phi_1<\theta_y$ so that $R(\theta_x)=R(\theta_y)$ and the claim (\ref{eq:cos}) follows. 
\end{proof}

\begin{figure}
    \centering
    \includegraphics[width=1\linewidth]{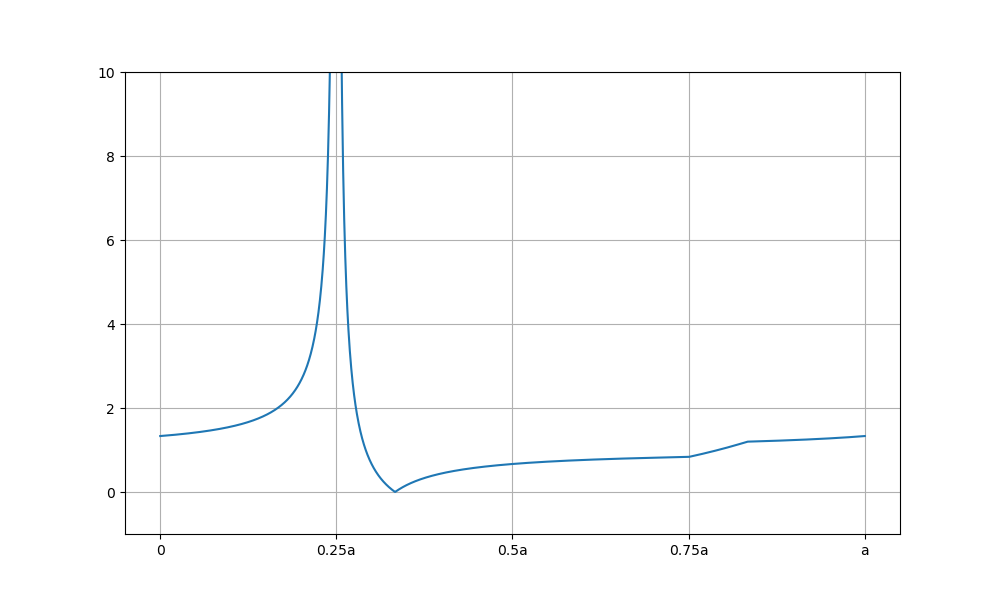}
    \caption{Plot of function $\theta\mapsto R(\theta)=\frac{sin(\rro_{\phi_1}(\theta))}{sin(\rro_{\phi_2}(\theta))}$ for $N=20$, $\phi_1=\frac{\pi}{30}$ and  $\phi_2=\frac{\pi}{40}=0.25a$.} 
    \label{figmin2}
\end{figure}

The previous result shows that the sorted coorbit embedding is not injective when $S$ has cardinal two. 
Next we show that the sorted coorbit representation provides an injective embedding in $\R^3$ with two or three generic windows. 

\begin{prop}\label{2dmaxfilter}
Consider the representation of the finite group $G=\Z_N\simeq\langle U_a \rangle =\{1,U_a,\ldots,U_a^{N-1}\}$ acting on $\Vc=\R^2$ by planar rotations of a multiple of $a=\frac{2\pi}{N}$. The sorted coorbit representation is injective in the following two cases.
\begin{enumerate}
\item Let $S=\{(1,1),(1,2),(1,3)\}$. Assume $w_1,w_2,w_3\in\R^2$ satisfy the following property: for any $k_1,k_2,k_3 \in\Z$, $\{ U_{k_1 a}w_1,U_{k_2 a}w_2,U_{k_3 a}w_3\}$ is a full spark frame. Then the max filter $\hat{\Phi}_{(w_1,w_2,w_3),S}:\hat{\Vc}\rightarrow\R^3$ is injective. 
\item Assume $N>2$ and let $S=\{(i,1),(j,1),(k,2)\}$, where $i,j,k\in[N]$, $i\neq j$ and, additionally, if $N$ is even then $i+j\neq N+1$. Assume $w_1,w_2\in\R^2$ satisfy the following property: for any $l_1,l_2\in\Z$, $\{ U_{l_1 a/2}w_1,U_{l_2 a/2}w_2\}$ is linearly independent. Then the map $\hat{\Phi}_{(w_1,w_2),S}:\hat{\Vc}\rightarrow\R^3$ is injective.
\end{enumerate}
\end{prop}

\begin{proof}
\mbox{}

1. Assume $x',y' \in \Vc$ so that $\Phi_{(w_1,w_2,w_3),S}(x')=\Phi_{(w_1,w_2,w_3),S}(y')$. Without loss of generality, assume $x'\neq 0$ because otherwise we obtain $y'=0$ and therefore $x'\sim y'$. 
Fix $x\in[x']$ and pick $y\in[y']$ so that $\ip{x}{y}= \max_{g \in G} \ip{x}{U^g_ay}>0$. Let $b\in[-\frac{a}{2},\frac{a}{2})$ and $\lambda>0$ be the unique scalars so that $x=\lambda\, U_b y$.  Replace $w_1,w_2,w_3$ with elements from their orbits so that $\Phi_{(w_1,w_2,w_3),S}(x)=(\ip{w_1}{x},\ip{w_2}{x},\ip{w_3}{x})$.  
	Then for some $g_1,g_2,g_3\in G$,
	\[\ip{w_1}{x}=\ip{U^{g_1}_a w_1}{y}~,~\ip{w_2}{x}=\ip{U^{g_2}_a w_2}{y}~,~\ip{w_3}{x}=\ip{U^{g_3}_a w_3}{y}.\]
 

If we assume that $x= \lambda U_by$ for some $b \in [0,\frac{a}{2})$ and $\lambda >0$. Then $g_1,g_2,g_3 \in \{0,{N-1}\}$.

If we assume that $x= \lambda U_by$ for some $b \in [-\frac{a}{2},0)$ and $\lambda >0$. Then $g_1,g_2,g_3 \in \{0,1\}$.

In both cases the pigeonhole principle implies that there exists $h \in \{0,1,{N-1}\}$ and $i\neq j \in \{1,2,3\}$ such that 
\[\ip{w_{i}}{x}=\ip{U^h_a w_i}{y}~\text{ and }~\ip{w_{j}}{x}=\ip{U^h_a w_j}{y}.\]

 Hence $y=u_a^h x$ and therefore $x\sim y$.

\quad

2. Without loss of generality assume $\norm{w_1}=\norm{w_2}=1$. 
\Cref{sortedcoobit} provides an explicit parametrization of $\Phi_{(w_1,w_2},S)$. \Cref{sortedcoobitfinal} yields: 

\begin{align*}
\Phi_{(w_1,w_2),S}(x) &= c_1 \begin{bmatrix} \cos(t_{1,1}-(-1)^i\lfloor\frac{i}{2}\rfloor a) \\ \cos(t_{1,1}-(-1)^j\lfloor\frac{j}{2}\rfloor a) \\ \cos(t_{1,2}-(-1)^k\lfloor\frac{k}{2}\rfloor a)\end{bmatrix}, \\
\Phi_{(w_1,w_2),S}(y) &= c_2 \begin{bmatrix} \cos(t_{2,1}-(-1)^i\lfloor\frac{i}{2}\rfloor a) \\ \cos(t_{2,1}-(-1)^j\lfloor\frac{j}{2}\rfloor a) \\ \cos(t_{2,2}-(-1)^k\lfloor\frac{k}{2}\rfloor a)\end{bmatrix},
\end{align*}
with
\[ t_{u,v} = \rro_{\phi_v}(\theta_u) := \min(|\theta_u-\phi_v|, a-|\theta_u-\phi_v|),~~ u,v\in\{1,2\} \]
and $c_1=\norm{x}$, $c_2=\norm{y}$, 
$\theta_1,\theta_2,\phi_1,\phi_2\in[0,a)$ so that $x=c_1 U_{\theta_1}e_1$, $y=c_2 U_{\theta_2}e_1$, $w_1=U_{\phi_1}e_1$, $w_2=U_{\phi_2}e_1$, $e_1=(1,0)^T$.

\ignore{

For this case we use the complex representation of this action by identifying $\Vc=\mathbb{R}^2$ with the field of complex numbers $\mathbb{C}$.

To simplify the proof we assume $\norm{w_1}=\norm{w_2}=1$. Furthermore, since the injectivity problem is rotation invariant, without loss of generality, we consider the case $w_1=e^{i\phi_1}$ and $w_2=1$ (so $\phi_2=0$). We represent $x$ and $y$ as $x=c_1e^{i\theta_1}$ and $y=c_2e^{i\theta_2}$. The phases are chosen to be the unique representatives in $[0, a)$ where $a = 2\pi/N$. The scalar product is given by $\ip{u}{v}=\text{Real}(u\bar{v})$. For the embedding to be injective we require $\phi_1\in(0,a)\setminus \{a/2\}$.

Let us define the shortest circular distance between $\theta_i$ and $\phi_j$ as:
\[ t_{i,j} = \rro_{\phi_j}(\theta_i) := \min(|\theta_i-\phi_j|, a-|\theta_i-\phi_j|). \]

With this notation, and using the structure from Proposition 4.1 for the measurement set $S=\{(1,1), (2,1), (1,2)\}$, the embedding $\Phi_{(w_1,w_2),S}$ is given by its components:
\begin{align*}
\Phi_{(w_1,w_2),S}(x) &= c_1 \begin{bmatrix} \cos(t_{1,1}) \\ \cos(a-t_{1,1}) \\ \cos(t_{1,2}) \end{bmatrix}, \\
\Phi_{(w_1,w_2),S}(y) &= c_2 \begin{bmatrix} \cos(t_{2,1}) \\ \cos(a-t_{2,1}) \\ \cos(t_{2,2}) \end{bmatrix}.
\end{align*}

}

Assume $\Phi_{(w_1,w_2),S}(x)=\Phi_{(w_1,w_2),S}(y)$. 
If $\Phi_{(w_1,w_2),S}(x)=0$ then $c_1=0$, which implies $x=0$. Similarly, $y=0$. Hence $x=y$.

Consider now the case $\Phi_{(w_1,w_2),S}(x) = \Phi_{(w_1,w_2),S}(y) \neq 0$. Hence $c_1,c_2\neq 0$. From the first two components of $\Phi_{(w_1,w_2),S}$ we get:
$$ c_1 \cos(t_{1,1}-(-1)^i\lfloor\frac{i}{2}\rfloor a) = c_2 \cos(t_{2,1}-(-1)^i\lfloor\frac{i}{2}\rfloor a)$$ and $$c_1 \cos(t_{1,1}-(-1)^j\lfloor\frac{j}{2}\rfloor a) = c_2 \cos(t_{2,1}-(-1)^j\lfloor\frac{j}{2}\rfloor a). $$
Let $k_1= -(-1)^i\lfloor\frac{i}{2}\rfloor$ and $k_2=-(-1)^i\lfloor\frac{i}{2}\rfloor$.
Notice that when $N$ is odd or, when $N$ is even and $i+j\neq N+1$, $(k_2-k_1)a\neq \pm \pi$. In fact, only when $N=2p$ is even and $i+j=N+1$ then  $k_2-k_1=\pm p$ and $(k_2-k_1)a=(-1)^i \pi$. It follows that either $c_1\cos(t_{1,1}+k_1a)\neq 0$ or $c_1\cos(t_{1,1}+k_2a)\neq 0$. 

Assume $c_1\cos(t_{1,1}+k_2a)\neq 0$. If this is not the case, switch $i$ with $j$ and repeat the argument. 
Dividing these two equations yields,
\begin{equation}
    \label{eq:St1}
\frac{\cos(t_{1,1}+k_1 a)}{\cos(t_{1,1}+k_2 a)} = \frac{\cos(t_{2,1} +k_1 a)}{\cos(t_{2,1}+k_2 a)}.
\end{equation}

Let $$S_{k_1,k_2}(x)= \frac{\cos(k_1a+x)}{\cos(k_2a+x)}.$$

Then equation (\ref{eq:St1}) turns into $S_{k_1,k_2}(t_{1,1})=S_{k_1,k_2}(t_{2,1})$. We claim $S_{k_1,k_2}$ is injective on $[0,\frac{a}{2}]$ (possibly excluding one point where it becomes singular) and therefore $t_{1,1}=t_{2,1}$.

Notice 
$$\frac{d S_{k_1,k_2}(x)}{dx}= \frac{\sin((k_2-k_1)a)}{(\cos(k_2a+x))^2}.$$
Therefore, $\frac{d S_{k_1,k_2}(x)}{dx}$ always has the same sign as $\sin((k_2-k_1)a)\neq 0$, which is independent of $x$.

\textbf{Case 1:} If $\cos(k_2a+x)$ does not vanish anywhere for $x\in[0,a/2]$, then $S_{k_1,k_2}$ is strictly monotone (increasing or decreasing) and therefore injective. 

\textbf{Case 2:}
 If $\cos(k_2a+x)$ vanishes somewhere on $[0,a/2]$, then there exists a unique $x_0 \in [0,a/2]$ such that $\cos(k_2a+x_0)=0$. Note that $x_0= \pi/2-k_2a$ or $x_0= -k_2a-\pi/2$. 
 
 If $x_0=0$ or $x_0=\frac{a}{2}$ then $S_{k_1k_2}$ is strictly monotone on $(0,\frac{a}{2})$ and therefore is injective.

Consider now the sub-case $0 <x_0 <\frac{a}{2}$. Thus $S_{k_1,k_2}$ is continuous and strictly monotone on $[0,x_0)$ and $(x_0,\frac{a}{2}]$, separately. We claim that $S_{k_1,k_2}$ is injective on $[0,x_0)\cup (x_0,\frac{a}{2}]$. Compute
\begin{align*}
S_{k_1,k_2}(0) - S_{k_1,k_2}\left(\frac{a}{2}\right) &= \frac{\cos(k_1a)}{\cos(k_2a)} - \frac{\cos\left(k_1a+\frac{a}{2}\right)}{\cos\left(k_2a+\frac{a}{2}\right)} \\
&= \frac{\sin((k_1 - k_2)a)\sin\left(\frac{a}{2}\right)}{\cos(k_2a)\cos\left(k_2a+\frac{a}{2}\right)}.
\end{align*}
Note that $\cos(k_2a)$ and $\cos\left(k_2a+\frac{a}{2}\right)$  have opposite signs and that $\sin\left(\frac{a}{2}\right)$ is always positive. Therefore, the sign of $S_{k_1,k_2}(0) - S_{k_1,k_2}\left(\frac{a}{2}\right)$ is the same as the sign of $\sin((k_2 - k_1)a)$ and the same as the sign of $\frac{d S_{k_1,k_2}(x)}{dx}$ at all $x\in[0,\frac{a}{2}]\setminus\{x_0\}$. 

Now we have two sub-cases:

\textbf{Sub-case 2.1:}
$S_{k_1,k_2}(x)$ is strictly increasing on $[0,x_0)$, strictly increasing on $(x_0,a/2]$ and $S_{k_1,k_2}(a/2) < S_{k_1,k_2}(0)$.

\textbf{Sub-case 2.2:}
$S_{k_1,k_2}(x)$ is strictly decreasing in $[0,x_0)$, strictly decreasing in $(x_0,a/2]$  and $S_{k_1,k_2}(a/2) > S_{k_1,k_2}(0)$.

In both cases we conclude that $S_{k_1,k_2}(x)$ is injective on $[0,x_0) \cup (x_0,a/2]$.

Therefore equation (\ref{eq:St1}) implies $t_{1,1}=t_{2,1}$ and $c_1=c_2$.

The equality involving the third component of $\Phi_{(w_1,w_2),S}$ implies $\cos(t_{1,2}+k_3a)=\cos(t_{2,2}+k_3a)$, where $k_3=-(-1)^k \lfloor\frac{k}{2}\rfloor$. Note that $t_{1,2},t_{2,2}\in[0,\frac{a}{2}]$ and $|k_3a|\leq \pi$. This implies $t_{1,2} = t_{2,2}$.

We have shown $\rro_{\phi_1}(\theta_1)=\rro_{\phi_1}(\theta_2)$ and $\rro_{\phi_2}(\theta_1)=\rro_{\phi_2}(\theta_2)$. This means 
\begin{eqnarray}
\label{eq:rhorho}    
\min(|\theta_1-\phi_1|, a-|\theta_1-\phi_1|) & = & \min(|\theta_2-\phi_1|, a-|\theta_2-\phi_1|), \\
\nonumber
\min(|\theta_1-\phi_2|, a-|\theta_1-\phi_2|) & = & \min(|\theta_2-\phi_2|, a-|\theta_2-\phi_2|).
\end{eqnarray}
To finish the proof, we now need to show that $\theta_1=\theta_2$ and therefore $x=y$.

Notice that (\ref{eq:rhorho}) can be interpreted as $\rro_{\theta_1}(\phi_1)=\rro_{\theta_2}(\phi_1)$ and $\rro_{\theta_1}(\phi_2)=\rro_{\theta_2}(\phi_2)$. 
As seen in Figure \ref{figmin1}, there are only two possible intersections for two distinct functions $\rro$: $\phi_1=\phi_2$ or $|\phi_1 - \phi_2|=\frac{a}{2}$. But either case contradicts the assumption that $\{U_{l_1a/2}w_1,U_{l_2a/2}w_2\}$ is linearly independent for every integers $l_1,l_2$. Therefore, $\rro_{\theta_1}(\cdot)=\rro_{\theta_2}(\cdot)$, i.e., $\theta_1=\theta_2$. This concludes the proof of \Cref{2dmaxfilter}.

\end{proof}

Thus, for the finite group of planar rotations, we conclude that the minimal dimension of an injective embedding using sorted coorbits is $m=3$. 
\vspace{5mm}

\begin{rem}
In the Case 1 of \cref{2dmaxfilter}, it is possible to estimate both the lower and the upper Lipschitz constants of the embedding $\Phi_{\bw, S}$. Specifically, since we have shown that the \emph{max filter} is achieved by the same group elements for at least two distinct windows $w_i, w_j$ with $i\neq j \in [3]$, it follows that
\[
a \cdot \dis(x, y) \le \|\Phi_{\bw, S}(x) - \Phi_{\bw, S}(y)\| \le b \cdot \dis(x, y),
\]
where
\[
a = \min_{i, j \in [3]} \sigma_1[w_i \mid w_j],
\quad \text{and} \quad
b = \max_{i, j \in [3]} \sigma_2[w_i \mid w_j].
\]

Here, $\sigma_1[w_i \mid w_j]$ and $\sigma_2[w_i \mid w_j]$ denote the smallest and largest singular values, respectively, of the matrix $[w_i|w_j]$.

Therefore the distortion is bounded above by $d=\frac{\max_{i, j \in [3]} \sigma_2[w_i \mid w_j]}{\min_{i, j \in [3]} \sigma_1[w_i \mid w_j]}$.
\end{rem}

\begin{rem}
It is likely that Case 1 of \cref{2dmaxfilter} holds even for different measurements from each sorted coorbit. That is, for $S=\{(i,1),(j,2),(k,3)\}$. We leave this case open for future study.
\end{rem}

\subsection{Analysis of the algebraic indices $\rho_n(q)$ and of \Cref{theoremmain2}}

Recall 
\[ H_n=\{(m_1,\ldots,m_n)\} \in \Z_N^n~,~(m_1,\ldots,m_n)~\text{distinct}.
\]

For $g=(g_1,\ldots,g_n),h=(h_1,\ldots,h_n)\in H_n$ and $q\in\N$,
\[ 
  \Gamma_{g,h}^q
  =
  \bigl\{\, (x,y)\in \mathbb{R}^2 \times \mathbb{R}^2 
  :
  \dim\!\bigl(\mathrm{span}\{\;U_{a}^{g_1} x - U_{a}^{h_1} y,\dots,\;U_{a}^{g_n}x - U_{a}^{h_n}y\}\bigr) = q 
  \bigr\}.
\]

\[
  \gamma_{g,h}^q  =
  \dim\bigl(\Gamma_{g,h}^q\bigr)\quad , \quad
  \rho_n(q) = 
  \max_{g,h \in H_n}\,\gamma_{g,h}^q.
\]
where by convention the dimension of the empty set is $-1$, i.e.,  
$\Gamma_{g,h}^q = \varnothing\quad\Longrightarrow\quad
  \gamma_{g,h}(q) =-1$.
Therefore, $\rho_{n}(q)=-1$ if $q>n$.

\begin{prop}\label{lambdarot}

The following holds.
\[
  \rho_1(q)=
  \begin{cases}
    2, & q=0,\\[6pt]
    4, & q=1,\\[6pt]
    -1, & q \ge 2.
  \end{cases}~~,~~
  \rho_2(q)\;=\;
  \begin{cases}
    2, & q=0,\\
    3, & q=1~\&~ \text{$N$ odd},\\
    4, & q=1~\&~\text{$N$ even},\\
    4, & q=2 \\
    -1, & q\ge3.
  \end{cases}
\]
\end{prop}

\begin{proof}
First, note that if $q> min(n,2)$ then $\rho_{n}(q)=-1$ since $\Gamma_{g,h}^q = \emptyset$


\begin{itemize}
    \item Case $(n=1,q=0)$. This condition is satisfied only by pairs $(x,y)$ of the form $(x,U_a^{-g_1}U_a^{h_1}x)$, with $x\in\R^2,$ which generate a 2-dimensional subspace in $\R^4$. Consequently, $\rho_{1}(0) = 2$. 
    
   \item Case $(n=1,q=1)$. Note that the complement of the set of pairs $(x,y)$ that satisfy condition (n = 1, q = 0), satisfies condition (n = 1, q = 1), that is, $(\Gamma^{0}_{g,h})^c=\Gamma^1_{g,h}$. Therefore, $\rho_1(1)=4$.
   
   \item Case $(n=2,q=0)$. This condition is satisfied by pairs $(x,y)$ that obey $U_a^{g_1}x=U_a^{h_1}x$ and $U_a^{g_2}x = U_a^{h_2}x$.     
   If $U_a^{-g_1}U_a^{h_1}=U_a^{-g_2}U_a^{h_2}$ then $\gamma^0_{g,h}=2$. Otherwise, $\Gamma_{g,h}^0=\{0\}$ and $\gamma^0_{g,h}=0$. 
   Therefore, $\rho_2(0)=2$.
   
   \item Case $(n=2,q=1)$ and $N$ odd. This condition is satisfied by pairs $(x,y)$ so that $\lambda_1(U_a^{g_1}x-U_a^{h_1}y) = \lambda_2 (U_a^{g_2}x-U_a^{h_2}y)$ for some $\lambda_1,\lambda_2 \in \R$ with $\lambda_1^2+\lambda_2^2=1$. Since $g_1\neq g_2$,  $h_1\neq h_2$ and $N$ is odd, it follows that $\lambda_1 U_a^{g_1}-\lambda_2 U_a^{g_2}$ is always invertible and $\Gamma^1_{g,h}$ is the projection onto the second factor (removing $(0,0)$) of the total space of the vector bundle $\{(\lambda,x,(\lambda_1 U_a^{h_1}-\lambda_2 U_a^{h_2})^{-1}(\lambda_1U_a^{g_1}-\lambda_2 U_a^{g_2})x)~,~\lambda=(\lambda_1,\lambda_2)\in \R\P^1,x\in\R^2\}$ with the base manifold the 1-dimensional real projective space $\R\P^1$. Therefore, $\rho_2(1)=3$.
   
   \item Case $(n=2,q=1)$ and $N$ even. Again, $(x,y)$ must satisfy same condition as above: $\lambda_1(U_a^{g_1}x-U_a^{h_1}y) = \lambda_2 (U_a^{g_2}x-U_a^{h_2}y)$ for some $\lambda_1,\lambda_2 \in \R$ with $\lambda_1^2+\lambda_2^2=1$. Unlike the case $N$ odd, when $g_1+g_2=0$ in $\Z_N$ then $U_a^{g_1}-(-1)U_a^{g_2}=0$ which shows that $\Gamma^1_{g,g}=\R^2\times\R^4$ and thus $\rho_2(1)=4$. Note: for any $(g_1,g_2)\in\Z_N$ with $g_1\neq g_2$ and $g_1+g_2\neq 0$, $\gamma^1_{g,g}=3$. 
   
    \item Case $(n=2,q=2)$ Note that for $g_1\neq g_2$ and $g_1+g_2\neq 0$, $\Gamma^2_{g,g}$ contains an open ball centered at a pair $(x,y)$ so that $\{x-y,U_a^{g_1-g_2}(x-y)\} $ is linearly independent. Thus $\rho_2(2)=4$.
   
\end{itemize}

\end{proof}

\begin{prop}\label{prop4.4} \mbox{}
Assume $N$ is odd. For Zariski generic $w_1,w_2,w_3\in \R^2$ and for every $S\subset [N] \times [3]$ of size 4 (or more), the map $\hat{\Phi}_{(w_1,w_2,w_3),S}$ is injective.

\end{prop}

\begin{proof}
\mbox{}

We verify \Cref{eq5modified} for $(n_1,n_2,n_3)=(2,1,1)$. The only triplets $(q_1,q_2,q_2)\in[n_1]\times[n_2]\times[n_3]$ 
are (2,1,1) and (1,1,1). In both cases \Cref{eq5modified} is satisfied,
\begin{equation}
\max_{
q_1\in[n_1],\ldots,q_p\in[n_p]
}\left(
min_{i\in[p]}\rho_{n_i}(q_i) -(q_1+\ldots+q_p) \right)
\leq 0=d_G  .
\end{equation}

\end{proof}

\begin{rem}
    In the case $N$ is even, Proposition \ref{prop4.4} remains true for sets $S$ that additionally satisfy the same condition as in Case 2 of \Cref{2dmaxfilter}, namely that the two elements $(k_1,i),(k_2,i)\in S$ taken from the same sorted coorbit should have $k_1+k_2 \neq 1+N$. This claim is shown by a more careful analysis of the "bad set" $\Fc_S$ in (\ref{eq:badset}). Specifically, the inclusion
     $\Fc_S\subset \bigcup_{\bg,\bh \in (H_n)^p} \Fc_{\substack{\bg,\bh}}$ can be improved by $\Fc_S\subset \bigcup_{\bg,\bh \in \widetilde{H_{n_1}}\times\cdots\cdots\widetilde{H_{n_p}}} \Fc_{\substack{\bg^{-1},\bh^{-1}}}$ with $\widetilde{H_n}$ denoting the set of only those n-tuples of "realizable" group elements. 
     In the case $N$ is even and $S$ as above, $\widetilde{H_{1}}=
     \Z_N$ and $\widetilde{H_{2}}\subset\{(m_1,m_2)\in\Z_N\times\Z_N~,~m_1+m_2\neq 0~,~m_1\neq m_2\}$. In this case $\max_{g,h\in\widetilde{H_2}}\gamma^1_{g,h}=3$. 
\end{rem}

\printbibliography

\appendix
\section{Summary of semi-algebraic geometry results}
In this appendix we summarize basic results from semi-algebraic geometry. 
We follow the notations from  
\cite{bochnak2013real}. For a pedestrian exposition on semi-algebraic results, the reader is referred to
\cite{coste2000introduction}.

\subsection{Semi-algebraic sets and functions}

\begin{defin}[Definition 2.1.1 in \cite{bochnak2013real}]
    Let $B$ be a subset of $\R[X_1,\dots,X_n]$.
    Denote $Z(B)=\{x: \in \R^n:f(x)=0,~\forall f \in B\}.$ An algebraic set of $\R^n$ is the $Z(B)$ for some $B$ be a subset of $\R[X_1,\dots,X_n]$.
\end{defin}

A generalization of algebraic sets is found in semi-algebraic sets, which encompass polynomial inequalities in addition to algebraic equations.

\begin{defin}[Definition 2.1.4 in \cite{bochnak2013real}]
A subset $S$ of $\R^n$ is a semi-algebraic set, if it has the form.
\[\bigcup_{i=1}^s\bigcap_{j=1}^{r_i}\{x\in \R^n: f_{i,j} *_{i,j} 0\}\]
where, $f_{i,j} \in \R[X_1,\dots,X_n]$ and $*_{i,j}$ is either $<$ or $=$ for $i \in [s], j \in [r_i]$.
\end{defin}


Now we will state some results from  \cite{bochnak2013real} without proof.

\begin{theor}[Theorem 2.2.1 in \cite{bochnak2013real}]\label{Tarski}
Let $A$ be a semi-algebraic set of $\R^{n+1}$ and let $\pi: \R^{n+1} \to \R^{n}$ be the projection on the first $n$ coordinates. Then $\pi(A)$ is a semi-algebraic set of $\R^n$.
\end{theor}

A simple corollary of \Cref{Tarski}, shown by induction, is the following:

\begin{cor}\label{proj}
If $A$ is a semi-algebraic subset of $\R^{n+k}$, its image by the projection on the space of the first $n$ coordinates is a semi-algebraic subset of $\R^n$.
\end{cor}

\begin{prop}[Proposition 2.2.2 in \cite{bochnak2013real}]
   If $A$ is a semi-algebraic subset of $\R^n$, its closure and its interior in $\R^n$ are again semi-algebraic.
\end{prop}

Let $A \subset \R^m$ and $B \subset \R^n$ be  semi-algebraic sets. A mapping $f:A \to B$ is called {\em semi-algebraic} if its graph:
\[\Gamma_f=\{(x,y) \in A \times B: y=f(x)\}\]
is a semi-algebraic set of $\R^m \times \R^n$.

\begin{prop}[Proposition 2.2.6 and 2.2.7  in \cite{bochnak2013real}]\label{image}
\par 
\begin{enumerate}
\item 
If $A$ and $B$ are semi-algebraic sets and $f:A \to B$ is a regular rational mapping (all its coordinates are rational
fractions whose denominators do not vanish on $A$), then $f$ is also semi-algebraic.
\item The image and the inverse image of a semi-algebraic set through a semi-algebraic mapping are semi-algebraic.
\item The composition of two semi-algebraic mappings is semi-algebraic.
\item If $f:A  \to \R$ such that $f \geq 0$ on $A$ then $\sqrt{f}$ is a semi-algebraic function.
\end{enumerate}
\end{prop}

\begin{defin}[Definition 12.7.1. in \cite{bochnak2013real}] Let \( M \subset \mathbb{R}^n \) be a semi-algebraic set. Let \( \xi = (E, p, M) \) be an \( R \)-vector bundle of rank \( k \) over \( M \). A family of local trivializations
\[
(U_i, \varphi_i : U_i \times \mathbb{R}^k \to p^{-1}(U_i))_{i \in I}
\]
of \( \xi \) is said to be a \textit{semi-algebraic atlas} of \( \xi \) if \( (U_i)_{i \in I} \) is a finite open semi-algebraic covering of \( M \) and the mappings
\[
\varphi_i^{-1} \circ \varphi_j |_{(U_i \cap U_j) \times \mathbb{R}^k}
\]
are continuous semi-algebraic, for every pair \( (i,j) \in I \times I \). Two semi-algebraic atlases are equivalent if their union is still a semi-algebraic atlas. A \textit{semi-algebraic vector bundle} is a vector bundle \( \xi = (E, p, M) \) equipped with an equivalence class of semi-algebraic atlases.

Let \( (\xi, (U_i, \varphi_i)_{i \in I}) \) and \( (\xi', (U_j', \varphi_j')_{j \in J}) \) be two semi-algebraic vector bundles over \( M \). A morphism \( \psi : \xi \to \xi' \) of vector bundles is said to be a \textit{semi-algebraic morphism} if the mappings
\[
(\varphi_j')^{-1} \circ \psi \circ \varphi_i |_{(U_i \cap U_j') \times \mathbb{R}^k}
\]
are continuous semi-algebraic, for every pair \( (i,j) \in I \times J \). A section \( s \) of \( \xi \) is said to be a \textit{semi-algebraic section} if the mappings
\[
\varphi_i^{-1} \circ s|_{U_i} : U_i \to U_i \times \mathbb{R}^k
\]
are continuous semi-algebraic, for every \( i \in I \).

We allow vector bundles to have possibly different ranks over different semi-algebraically connected components of \( M \).
\end{defin}

\begin{prop}[Proposition 12.1.4.] Let
\[
E_{n,k} = \{(A,v) \in G_{n,k}(\mathbb{R}) \times \mathbb{R}^n \mid A \cdot v = v\},
\]
\[
E_{n,k}^\perp = \{(A,v) \in G_{n,k}(\mathbb{R}) \times \mathbb{R}^n \mid A \cdot v = 0\}.
\]
\textit{Let} \( p_{n,k} \) (\textit{resp.} \( p_{n,k}^\perp \)) \textit{be the canonical projection of} \( E_{n,k} \) (\textit{resp.} \( E_{n,k}^\perp \)) \textit{onto} \( G_{n,k}(\mathbb{R}) \). \textit{Then} \( \gamma_{n,k} = (E_{n,k}, p_{n,k}, G_{n,k}(\mathbb{R})) \) \textit{and} \( \gamma_{n,k}^\perp = (E_{n,k}^\perp, p_{n,k}^\perp, G_{n,k}(\mathbb{R})) \) \textit{are pre-algebraic vector bundles over} \( G_{n,k}(\mathbb{R}) \) \textit{of rank} \( k \) \textit{and} \( n-k \), \textit{respectively. Moreover,}
\[
\gamma_{n,k} \oplus \gamma_{n,k}^\perp \cong_{\text{alg}} \mathbb{R}^n.
\]
\textit{The bundle} \( \gamma_{n,k} \) \textit{is called the universal vector bundle over} \( G_{n,k}(\mathbb{R}) \).
\end{prop}

\begin{defin}[Definition 12.1.6.] 
A pre-algebraic vector bundle over \(X\) is said to be algebraic if there exists an injective algebraic morphism from \(\xi\) to a trivial bundle \( \mathbb{R}^n_X \) (i.e., \(\xi\) is algebraically isomorphic to a pre-algebraic vector subbundle of a trivial bundle).
\end{defin}

\begin{theor}[Theorem 12.1.7.] 
Let \(\xi = (E, p, X)\) be a pre-algebraic vector bundle of rank \(k\) over \(X\). Then the following properties are equivalent:
\begin{enumerate}
    \item[(i)] \(\xi\) is algebraic.
    \item[(ii)] For every \(x \in X\), there exist global algebraic sections \(s_1, \ldots, s_k\) of \(\xi\) such that \(s_1(x), \ldots, s_k(x)\) generate the fibre \(p^{-1}(x)\) as an \(R\)-vector space.
    \item[(iii)] There exists a surjective algebraic morphism from a trivial bundle \( \mathbb{R}^n_X \) onto \(\xi\) (i.e., \(\xi\) is algebraically isomorphic to a pre-algebraic quotient of a trivial bundle).
    \item[(iv)] There exists a pre-algebraic vector bundle \(\xi'\) over \(X\) such that \(\xi \oplus \xi'\) is algebraically isomorphic to a trivial bundle \( \mathbb{R}^n_X \) (i.e., \(\xi\) is a pre-algebraic direct factor of a trivial bundle).
    \item[(v)] There exists a regular mapping \(f : X \rightarrow G_{n,k}(\mathbb{R})\) for some \(n \geq k\), such that \(\xi\) is algebraically isomorphic to \(f^*(\gamma_{n,k})\).
    \item[(vi)] There exists a projective module of finite type \(M\) over \(R(X)\) such that \(\operatorname{Calg}(\xi)\) is isomorphic to \(R \times \mathcal{O}_{X}(M)\).
\end{enumerate}
\end{theor}

\begin{prop}[Proposition 12.1.8.]
\begin{enumerate}
    \item[(i)] The vector bundles \(\gamma_{n,k}\) and \(\gamma_{n,k}^\perp\) over \(G_{n,k}(\mathbb{R})\) are algebraic.
    \item[(ii)] If \(\xi\) is an algebraic vector bundle over \(X\), and \(f : Y \to X\) is a regular mapping, then \(f^*(\xi)\) is algebraic.
    \item[(iii)] If \(\xi\) and \(\eta\) are algebraic vector bundles over \(X\), then \(\xi \oplus \eta\), \(\xi \otimes \eta\), \(\xi^\vee\), \(\wedge^q \xi\), and \(\operatorname{Hom}(\xi, \eta)\) are algebraic vector bundles.
\end{enumerate}
\end{prop}

\subsection{Dimension of semi-algebraic sets}

\begin{defin}[Definition 2.8.1 in \cite{bochnak2013real}]\label{dimdef}
Let  $A\subset \R^n$ be a semi-algebraic set. Denote be $P(A)=\R[X_1,\dots,X_n]/I(A)$ the ring of polynomial functions on $A$, where $I(A)$ is the ideal of polynomials that vanish on $A$.
The semi-algebraic dimension of $A$ denoted by $\dim(A)$ is the dimension of the ring $P(A)$, i.e. the maximal length of chains of prime ideals of $P(A)$. 
\end{defin}

\begin{prop}[Theorem 2.3.6 in \cite{bochnak2013real}]\label{PropA9}
Every semi-algebraic subset of $\R^n$ is the disjoint union of a finite number of semi-algebraic sets, 
each of them semi-algebraically homeomorphic to an open hypercube $(0,1)^{d}$ for some $d\in\N$ (with $(0,1)^0$ being a point).
\end{prop}
For our purposes we employ the following result as an equivalent definition for dimension of a semi-algebraic set. 
\begin{prop}[Corollary 2.8.9 in \cite{bochnak2013real}]\label{SemiDim}
Let $A \in \R^n$ be decomposed as the disjoint union of finitely many pieces which are semi-algebraically homeomorphic to open hypercubes $\{(0,1)^{d_i}\}_{i \in I}$. Then the dimension of $A$ is the maximum dimension of hypercubes, i.e. 
$\dim(A)= \max_{i \in I}d_i$.
\end{prop}

\begin{prop}[Proposition 2.8.5. II in \cite{bochnak2013real}]\label{cartesian}
Let $A$ and $B$ be two semi-algebraic sets.
Then $dim(A \times B)= \dim(A)+\dim(B)$.

\end{prop}

\begin{prop}[Proposition 2.8.6. in \cite{bochnak2013real}]\label{dimproject}
  Let $A$ be a semi-algebraic subset of \(\mathbb{R}^{n+1}\) and let 
  \(\Pi : \mathbb{R}^{n+1} \to \mathbb{R}^n\) be the projection onto the space 
  of the first \(n\) coordinates. Then 
  \[
    \dim\bigl(\Pi(A)\bigr) \;\le\; \dim(A).
  \]
\end{prop}

Finally two very important theorems of semi-algebraic geometry are the following:
    
\begin{theor}[Theorem 2.8.8 in \cite{bochnak2013real}]\label{alg2}
   Let $A$ be a semi-algebraic subset of $\R^n$, and $f:A \to R^k$ a semi-algebraic mapping (not necessarily continuous). Then $\dim f(A) \le \dim A$. 
\end{theor}

 \begin{theor}[Proposition 2.8.2 in \cite{bochnak2013real}]\label{semimap}
   Let $A \subset \R^n$ be a semi-algebraic set.
   Then \[\dim(A)= \dim(\clos(A))= \dim(\clos_{Zar}(A)).\]
   Note that $clos_{Zar}(A)$ is an algebraic set thus that proposition states that the semi-algebraic dimension of A is equal to the algebraic dimension of its Zariski closure.
 \end{theor}

Using \Cref{semimap} and the fact that an algebraic set  closed with respect to Zariski topology and of dimension strictly smaller than the dimension of the ambient space is nowhere dense, we obtain the following corollary.

\begin{cor}\label{nowheredense}
    Let $A \subset \R^N$ be a semi-algebraic subset of $\R^N$. If $\dim(A)<n$ then $A$ is nowhere dense.
\end{cor}

Moreover, note that any semi-algebraic set consists of finitely many connected components.

 \begin{theor}[Theorem 2.45 in \cite{bochnak2013real}] Every semi-algebraic set has finitely many connected components which are semi-algebraic. Every semi-algebraic set is locally connected.
 \end{theor}

\end{document}